\RequirePackage{fix-cm}
\documentclass[smallcondensed]{svjour3}     %
\smartqed  %
\usepackage{bm}
\usepackage{amsmath,amssymb,amsfonts,amssymb,mathtools}

\usepackage{graphicx}
\usepackage{caption,subcaption}
 \graphicspath{{./figures/}{./}}

\usepackage{etoolbox}
\newcommand*{\affaddr}[1]{#1}
\newcommand*{\affmark}[1][*]{\textsuperscript{#1}}

\usepackage{enumerate,verbatim}
\usepackage{xcolor}
\usepackage{hyperref}
\hypersetup{
	plainpages=false,
    colorlinks,
    linktocpage=true,   %
    linkcolor={red!50!black},
    citecolor={blue!50!black},
    urlcolor={blue!80!black}
}
\usepackage[hyperpageref]{backref}

\newtheorem{assumption}[theorem]{Assumption}

\makeatletter
\BeforeBeginEnvironment{proof}{%
  \def\@Opargbegintheorem#1#2#3#4{#4\trivlist
      \item[]{#3#2\@thmcounterend\ }}%
}
\AfterEndEnvironment{proof}{%
  \def\@Opargbegintheorem#1#2#3#4{#4\trivlist
      \item[\hskip\labelsep{#3#1}]{#3(#2)\@thmcounterend\ }}%
}
\makeatother

\newcommand{\bfa}{{\bf a}}

\newcommand{\bfb}{{\bf b}}

\newcommand{\bfc}{{\bf c}}

\newcommand{\bfE}{{\bf E}}

\newcommand{\bff}{{\bf f}}

\newcommand{\bfH}{{\bf H}}

\newcommand{\bfJ}{{\bf J}}

\newcommand{\bfn}{{\bf n}}

\newcommand{\bfp}{{\bf p}}

\newcommand{\bfq}{{\bf q}}

\newcommand{\bfS}{{\bf S}}
\newcommand{\bft}{{\bf t}}

\newcommand{\bfu}{{\bf u}}
\newcommand{\bfV}{{\bf V}}
\newcommand{\bfv}{{\bf v}}
\newcommand{\bfw}{{\bf w}}

\newcommand{\bfx}{{\bf x}}

\newcommand{\bfy}{{\bf y}}

\newcommand{\bfz}{{\bf z}}

\newcommand{\bfPi}{{\bf \Pi}}
\newcommand{\bfxi}{\boldsymbol{\xi}}
\newcommand{\bfeta}{\boldsymbol{\eta}}

\newcommand{\curl}{\operatorname{curl}}
\renewcommand{\div}{\operatorname{div}}
\newcommand{\bcurl}{\bm{{\operatorname{curl}}}}

\newcommand{\dd}{\,{\rm d}}

\renewcommand{\div}{\operatorname{div}}

\newcommand{\snorm}[1]{\left\vert#1\right\vert}
\makeatletter
\newcommand{\tnorm}{\@ifstar\@tnorms\@tnorm}
\newcommand{\@tnorms}[1]{%
\left|\mkern-2.5mu\left|\mkern-2.5mu\left|
#1
\right|\mkern-2.5mu\right|\mkern-2.5mu\right|
}
\newcommand{\@tnorm}[2][]{%
\mathopen{#1|\mkern-2.5mu#1|\mkern-2.5mu#1|}
#2
\mathclose{#1|\mkern-2.5mu#1|\mkern-2.5mu#1|}
}
\makeatother

\newcommand*{\medcup}{\mathbin{\scalebox{1.5}{\ensuremath{\cup}}}}%

\newcommand{\LC}[1]{\textcolor{black}{#1}}
\newcommand{\SC}[1]{\textcolor{black}{#1}} %

\newcommand{\RG}[1]{\textcolor{black}{#1}} %
\newcommand{\newrevision}[1]{\textcolor{black}{#1}}

\numberwithin{theorem}{section}
\numberwithin{equation}{section}

\begin{document}

\title{Immersed Virtual Element Methods for Elliptic Interface Problems in Two Dimensions\thanks{This work was funded in part by NSF grants DMS-1913080, DMS-2012465, and DMS-2136075.}
}

\titlerunning{Immersed Virtual Element Methods}        %

\author{Shuhao Cao\affmark[1] \and
        Long Chen\affmark[2] \and
        Ruchi Guo\affmark[2] \and
        Frank Lin\affmark[2]
}

\authorrunning{S. Cao, L. Chen, R. Guo and F. Lin} %

\institute{S. Cao \at
              \affaddr{ \affmark[1]Department of Mathematics and Statistics, University of Missouri Kansas City, Kansas City, MO} \\
              \email{sch59@umkc.edu}           %
           \and
           L. Chen \and R. Guo \and F. Lin \at
              \affaddr{ \affmark[2]Department of Mathematics, University of California Irvine, Irvine, CA}\\
              \email{chenlong@math.uci.edu}, \email{ruchig@uci.edu}, \email{fmlin@uci.edu}
}

\date{Received: date / Accepted: date}

\maketitle

\begin{abstract}
This article presents an immersed virtual element method for solving a class of interface problems that combines the advantages of both body-fitted mesh methods and unfitted mesh methods. A background body-fitted mesh is generated initially. On those interface elements, virtual element spaces are constructed as solution spaces to local interface problems, and exact sequences can be established for these new spaces involving discontinuous coefficients. The discontinuous coefficients of interface problems are recast as Hodge star operators that are the key to project immersed virtual functions to classic immersed finite element (IFE) functions for computing numerical solutions. An a priori convergence analysis is established robust with respect to the interface location. The proposed method is capable of handling more complicated interface element configuration and provides better performance than the conventional penalty-type IFE method for the $\bfH(\curl)$-interface problem arising from Maxwell equations. It also brings a connection between various methods such as body-fitted methods, IFE methods, virtual element methods, etc.
\end{abstract}

\section{Introduction}
\label{sec:intro}

Interface problems widely appear in many engineering and physical applications involving multiple materials or media that incorporate discontinuous coefficients for the related partial differential equations (PDEs). For example, Figure \ref{fig:domain} illustrates a two-dimensional bounded domain $\Omega$ that is formed by two different materials separated by a closed smooth curve $\Gamma \in C^{1,1}$, i.e., $\Gamma$ separates $\Omega$ into subdomains $\Omega^{+}$ and $\Omega^{-}$ such that $\overline{\Omega}=\overline{\Omega^{+}\cup \Omega^{-}\cup \Gamma}$. The main challenge of using standard finite element methods (FEMs) is that solutions of interface problems are not smooth across the interface.
It is well known that FEMs can be used to solve interface problems with optimal accuracy~\cite{Xu1982a,Chen.Z;Zou.J1998,2008RainaldJuanFernando,2010LiMelenkWohlmuthZou,2016ZhengLowengrub} based on body-fitted and shape regular meshes. The ``body-fittedness'' refers to that the interface is well approximated by edges of 
elements~\cite{2010LiMelenkWohlmuthZou}, i.e., the piecewise linear approximated interface cannot intersect any element interior. However, it is nontrivial and time-consuming to generate such a shape regular mesh that fits the interface, as it generally requires certain global modifications. This issue will become more severe for complex geometry or moving interface problems, especially in three dimensions.

\begin{figure}[htbp]
\begin{center}
\includegraphics*[width = 4.5 cm]{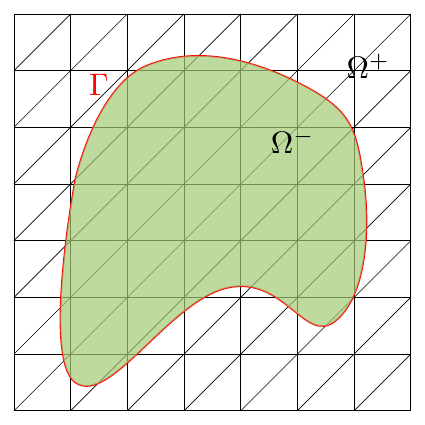}
\caption{The domain and a uniform triangulation for the interface problem}
\label{fig:domain}
\end{center}
\end{figure}

So, it becomes critical for the purpose of efficiency to relax the mesh restriction for interface problems. Generally speaking, two different groups of methods can be found in this field: (i) modify the finite element spaces or finite difference stencils to encode the jump conditions into the discretization; (ii) modify the mesh only near the interface and then apply either continuous or discontinuous Galerkin formulation.

The first approach employs meshes that are completely independent of the interface, i.e., the so-called unfitted mesh methods. As the mesh cannot resolve interface geometry, special treatments are needed on interface elements. 
The Nitsche's idea~\cite{1971Nitsche} uses penalties to enforce the jump condition, see e.g., CutFEM~\cite{2015BurmanClaus} or unfitted FEM~\cite{2002HansboHansbo}. Another strategy is to construct special FEM functions or finite difference stencils on interface elements, such as the immersed interface method~\cite{1994LevequeLi}, the MIB method~\cite{2007YuZhouWei}, the multiscale FEM~\cite{2010ChuGrahamHou}, and the immersed finite element (IFE) methods~\cite{li2004immersed,2015LinLinZhang,guzman2017finite} to be discussed. 
In particular, for the IFE method, a set of local basis functions on interface elements are devised as piecewise polynomials that include jump conditions in their connection in a pointwise or an averaging 
sense. The convergence of IFE methods for $H^1$ interface problems have been established in~\cite{li2004immersed,2015LinLinZhang,guzman2017finite} and improved recently in~\cite{Guo;Lin:2019immersed,guo2019improved,2021JiWangChenLiA,2021JiWangChenLiB}. These methods still obtain the optimal convergence order where the hidden constant is independent of the interface location relative to the mesh. \RG{However, for almost all these unfitted mesh methods, the approximation spaces are not conforming. The non-conformity actually becomes an essential issue for solving $\bfH(\curl)$ interface problems arising from Maxwell equations, which has been widely discussed in the literature~\cite{2001BenBuffaMaday,2016CasagrandeHiptmairOstrowski,2016CasagrandeWinkelmannHiptmairOstrowski,2020GuoLinZou,2008HuShuZou}, also see the discussion below. It is one of the motivation for the proposed method that aims to develop a method based on conforming approximation spaces.}

\RG{For the second approach, as the modification is only performed locally near the interface, shape regularity, in general, cannot be achieved. Instead, the refined meshes are required to satisfy the maximum angle condition~\cite{1976BabuskaAziz,2020KobayashiTsuchiya,1992Michal} to obtain optimal convergence rates robust with respect to element shapes. One work in this direction can be found in~\cite{2009ChenXiaoZhang}. This is indeed achievable for the 2D case, as the maximum angle condition can be always satisfied for arbitrary interface location \cite{2021CaoChenGuo}, and even for adaptive meshes \cite{WeiChenHuangEtAl2014Adaptive}.
However, such a local triangulation satisfying the maximum angle condition might not be readily available or requires strenuous effort to generate in the 3D case \cite{Edelsbrunner2000Triangulations,LiTeng2001Generating,MooreSaigal2005Eliminating}. This obstacle also motivates us to develop a method that does not rely on a local triangulation. Even though the current work is only for the 2D problems, it can shed light on the 3D case. In fact, we have recently established the 3D IVE spaces in \cite{2022CaoChenGuo} after this work. }

Recently, the authors in~\cite{chen2017interface} proposed a novel method that directly works on polygonal or polyhedral elements cut from interface, instead of re-triangulating them to simplices. The key of \cite{chen2017interface} is to employ directly the virtual element method (VEM) on these elements for the discretization, on which only degrees of freedom (DoFs) are necessary for assembling the final linear system,
e.g., see~\cite{2014VeigaBrezziMariniRusso,2013BeiraodeVeigaBrezziCangiani,2016VeigaBrezziMarini,2020BeiroMascotto} and the reference therein. 
The ``virtual'' shape functions, which are $H^1$ functions that serve as the solutions to certain local problems but do not need to be explicitly solved, are then projected to polynomial spaces for computation through DoFs. One key advantage is its flexibility for element shapes being polygonal or polyhedral. 
As the interface may intersect elements arbitrarily which generates elements with high aspect ratio, for the aforementioned approach in~\cite{chen2017interface}, one major difficulty is to obtain a robust a priori error estimate independent of the potential anisotropic subelement shapes. Some anisotropic error analysis of VEM can be found in~\cite{2018CaoChen,Cao;Chen:2018AnisotropicNC,2021CaoChenGuo} for different interface problems.

Inspired by VEM~\cite{chen2017interface} and IFE methods in the literature, is it possible for a numerical method to take both the advantages of conformity provided by virtual element spaces and robust optimal approximation capabilities of IFE spaces? The question severs as one major motivation for this work. For this purpose, we shall develop $H^1$, 
$\bfH(\curl)$ and $\bfH(\div)$ virtual element spaces involving discontinuous coefficients, i.e., they are solution spaces to some local interface problems incorporating jump conditions related to the underlying equations. As the interface is immersed into the design of the virtual element spaces, we shall call it {\em immersed \LC{(interface)} virtual element method} (IVEM). 
The key idea is to use the conforming virtual element spaces on a shape regular background unfitted mesh $\mathcal T_h$ for discretization, and then to project them to the IFE spaces on interface elements which are cut by the interfaces from the background mesh. \LC{The virtual element space provides the conformity and the IFE space can offer sufficient and robust approximation capabilities locally. We also note that this practice exhibits similarity to the Trefftz finite element method (Trefftz-FEM)~\cite{JirousekWroblewski1996T}, in which the basis functions are fundamental solutions to certain local problems. Another resemblance is that Trefftz FEM may relax the exact inter-element continuity to yield a ``quasi-conforming'' discretization \cite{IhlenburgBabuska1995Dispersion,2020AnandOvallEtAlTrefftz}, which carries the same spirit with the IFE spaces locally.} Moreover, as subelements of elements are treated together through the jump conditions instead of independently as anisotropic polygons, it is highlighted that the coercivity can be established of which the hidden constants are independent of subelement shapes. This property does not hold for virtual element spaces defined only on subelements~\cite{2018CaoChen,chen2017interface}, where the coercivity constant may depend on the anisotropy of polygons or polyhedra, and refined analysis is needed to establish a robust error analysis. 

In particular, we will consider the following \LC{$H^1$ and $\bfH(\curl)$ interface problems in two dimensions} and refer to~\cite{2021Ji} for $\bfH(\div)$ interface problems. \RG{Due to the fact that solution exhibiting low regularity near the interface, especially for the $\bfH(\curl)$ equations \cite{1999MartinMoniqueSerge,2004CostabelDaugeNicaise}, in this work we only consider the lowest order methods.} The first problem of interest is an $H^1$-elliptic interface problem
\begin{equation} 
\label{eq:problem-pde-interface}
\begin{aligned}
      -\nabla\cdot(\beta\nabla u)&=f \quad \text{ in }  \Omega^-  \cup \Omega^+, \\
      u&=0 \quad  \text{ on } \partial \Omega,
\end{aligned}
\end{equation}
with $f\in L^2(\Omega)$, and the continuity and flux jump conditions
\begin{subequations}
\label{eq:jump}
\begin{align}
       [u]_{\Gamma} & := u^+ - u^- = 0, \\
  [\beta \nabla u \cdot \bfn]_{\Gamma} &:=  \beta^{+}\nabla u^{+}\cdot \mathbf{ n}  - \beta^{-}\nabla u^{-}\cdot \mathbf{ n} =0 ,
\end{align}
\end{subequations}
where $\mathbf{n}:= \bfn(\bfx)$ denotes the unit normal vector to $\Gamma$ at $\bfx=(x_1,x_2)\in \Gamma$ pointing from $\Omega^-$ to $\Omega^+$. In the following discussion, $\bfn$ always denotes the unit outward normal vector, and $\bft$ denotes the tangential vector which is a counterclockwise rotation of $\bfn$ by $\pi/2$.

The second model we are interested in is an $\bfH(\text{curl})$ interface problem arising from Maxwell equations: 
\begin{subequations}
\label{model}
\begin{align}
\label{inter_PDE}
\bcurl~(\alpha \, \text{curl}~\bfu) + \beta \bfu &= \bff & \text{ in } &  \Omega^-  \cup \Omega^+, \\
\bfu\cdot\bft &= 0  & \text{ on } &  \partial\Omega, 
\end{align}
\end{subequations}
with $\bff\in\bfH(\div; \Omega)$, 
where the operator $\text{curl}$ is for vector functions $\bfv=(v_1,v_2)^{\intercal}$ such that $\text{curl}~\bfv=\partial_{x_1}v_2 - \partial_{x_2}v_1$ while $\bcurl$ is for scalar functions $v$ such that $\bcurl~v = \left ( \partial_{x_2}v, - \partial_{x_1}v \right )^{\intercal}$ with ``$^{\intercal}$" denoting the transpose herein. The following jump conditions at the interface $\Gamma$ are imposed:
\begin{subequations}
\label{inter_jc}
\begin{align}
[\bfu\cdot\bft]_{\Gamma} &:=  \bfu^+\cdot\bft -  \bfu^-\cdot\bft = 0,  \label{inter_jc_1} \\
[\alpha \,\text{curl}~\bfu]_{\Gamma} &:=  \alpha^+ \text{curl}\,\bfu^+ - \alpha^- \text{curl}\,\bfu^-   = 0,
\label{inter_jc_2} \\
[\beta \bfu\cdot\bfn ]_{\Gamma}&: = \beta^+ \bfu^+ \cdot\bfn - \beta^- \bfu^- \cdot\bfn = 0. 
\label{inter_jc_3}
\end{align}
\end{subequations}
\RG{In equations} \eqref{eq:problem-pde-interface} and \eqref{model}, the coefficients $\alpha$ and $\beta$ in $\Omega$ are assumed to be positive piecewise constant functions of which the locations of the discontinuity align with one another:
$$
 \alpha(x,y)= \begin{cases} 
      \alpha^{+}, & (x,y)\in \Omega^{+}, \\
      \alpha^{-}, & (x,y)\in \Omega^{-}, \\
   \end{cases}
   \quad\quad\quad
 \beta(x,y)= \begin{cases} 
      \beta^{+}, & (x,y)\in \Omega^{+}, \\
      \beta^{-}, & (x,y)\in \Omega^{-}. \\
   \end{cases}
$$
Note that the two models above share the same parameter $\beta$ which can be interpreted from the perspective of de \RG{Rham complexes}.  
The proposed virtual element spaces can inherit this kind of structure on each interface element.

Similar to the standard virtual element spaces in the literature, our new $H^1$, $\bfH(\curl)$ and $\bfH(\div)$ virtual element spaces admit the nodal and edge DoFs which make them conforming in their respective Sobolev spaces even with the presence of interface-cutted mesh and  discontinuous parameters. These DoFs also enable us to establish the global exact sequence, and e.g., the following commutative diagrams 
\begin{equation}
\label{intro:DR_curl}
\left.
\begin{array}{ccccccc}
\mathbb R \xrightarrow[]{\quad} &     \LC{H^2(\beta;\mathcal T_h)} & \xrightarrow[]{~~\nabla~~} & \bfH^1(\text{curl},\alpha,\beta;\mathcal T_h) & \xrightarrow[]{~~\curl~~} & H^1(\alpha;\mathcal T_h)  &\xrightarrow[]{\quad} 0  \\
~~~~& \quad \bigg\downarrow I^n_{h}  &  & ~~~~\bigg\downarrow I^e_{h}  &  &~~~~\bigg\downarrow \pi^{\alpha_h}_{h} \\
\mathbb R \xrightarrow[]{\quad}& V^n_{h} & \xrightarrow[]{~~\nabla~~} & \bfV^e_{h} & \xrightarrow[]{~~\curl~~} & Q^{\alpha_h}_h  &\xrightarrow[]{\quad} 0.
\end{array}\right.
\end{equation}
See Sections \ref{subsec:SpaceNorms} and \ref{subsec:IVE} for definitions of spaces and operators.

Constructing special shape functions by solving local problems to capture certain behavior of solutions can be traced back to the fundamental work of Babu\v{s}ka et al. in~\cite{1994BabuskaCalozOsborn,1983BabuskaOsborn}. In particular, for a 1D case, the basis functions in~\cite{1994BabuskaCalozOsborn,1983BabuskaOsborn} are the solutions of
\begin{equation}
\label{1D_loc}
-(\beta(x) u_h')' = 0 \quad\quad \text{in} \; [a,b]
\end{equation}
subject to some boundary conditions at the ending points $a,b$. It could be considered as the local problems of VEM with variable coefficients. Due to the trivial 1D geometry, solutions of \eqref{1D_loc} can be expressed as $\int_a^x \beta^{-1}(s) \dd s$. 
When $\beta$ is a piecewise constant function, they become exactly the 1D IFE functions~\cite{1998Li}. 
Namely, for this case, the 1D VEM and IFE spaces are identical, but they are distinguished in higher dimensions due to more complicated geometry. From this point of view, on one hand, the proposed IVEM is a more straightforward generalization of the early approach of Babu\v{s}ka et al. On the other hand, the conventional IFE space is also important to provide robust local approximation capabilities, and thus is suitable for constructing projections.

We also note that the newly constructed $H^1$ virtual element space is similar to the multiscale finite element space in~\cite{2010ChuGrahamHou} in the sense that local interface problems are used to develop the approximation spaces. In both approaches, standard non-piecewise polynomials on interface elements cannot be used to approximate the solutions to these local interface problems due to the jump conditions across the interface. In~\cite{2010ChuGrahamHou}, the authors generate a local mesh and use standard finite element functions for approximation. 
Here we propose projecting the virtual element spaces to IFE spaces consisting of piecewise polynomials that can accurately capture the jump conditions. We will show that, similar to the conventional VEM, these projections are indeed computable directly through the DoFs.

The proposed method is not only a new formulation of IFE or VEM in the literature, but also inherits the advantages of both the two methods, or even the general fitted mesh and unfitted mesh methodology. First, it is still able to solve interface problems on a background unfitted mesh. However, different from most of the unfitted mesh methods aforementioned that do not impose any DoFs on edges or nodes associated with cutting points of interface, the proposed one does impose these newly added DoFs. With this property, it may better resolve the more complex geometry but without generating an extra triangulation near the interface. In other words, we use a virtual body-fitted mesh. Second, it is known that IFE shape functions satisfying certain DoFs are in general not easy to construct, and theoretically their existence are subject to some geometric conditions~\cite{Guo;Lin:2019immersed,2020GuoLinZou,2021Ji}. Within the VEM framework, this issue has been completely addressed, since the DoFs are imposed through virtual functions which always exist by solving local problems. Third, compared with the anisotropic analysis for conventional VEM~\cite{Brenner;Sung:2018Virtual,2018CaoChen}, the robust error analysis of the proposed method can be, thanks to the shape regularity of background meshes and properties of IFE spaces, easily and systematically obtained regardless of subelement shape. Finally, compared with other penalty-type methods in the literature~\cite{guo2019improved,2015LinLinZhang}, the proposed method requires only a locally computed edge term within each element, and thus makes the assembling procedure easier as the stabilization term does not need explicitly the interaction of neighbor elements' DoFs. 

One remarkable advantage of using the proposed method is to recover the optimal convergence for solving $H(\curl)$ interface problems on unfitted meshes. 
The $\bfH(\curl)$ equations are sensitive to the conformity of the approximated spaces due to its low regularity. 
Discontinuous Galerkin methods can obtain an optimal convergence, but this is based on the fact that the broken non-conforming space contains an $\bfH(\curl)$-conforming subspace \LC{when no interface is present}, see the analysis in \cite{2005HoustonPerugiaSchneebeli,2004HoustonPerugiaSchotzau,2005HoustonPerugiaDominik}. 
Unfortunately, many aforementioned conventional unfitted mesh methods do not preserve this property which may cause the loss of accuracy. This phenomenon has been numerically observed and theoretically proved in \cite{2016CasagrandeHiptmairOstrowski,2016CasagrandeWinkelmannHiptmairOstrowski} for Nitsche's penalty methods. In \cite{2020LiuZhangZhangZheng}, the authors assume a higher regularity, \RG{i.e., at least piecewise $H^2(\Omega^{\pm})$}, to overcome this issue. As for IFE methods, standard penalty-type methods still do not achieve optimal convergence, and a Petrov-Galerkin method can be applied, see~\cite{2020GuoLinZou}, and achieve optimal order convergence with certain conditions. The IVEM proposed in this paper is able to circumvent this issue since the underlying IVE space is always conforming which is distinguished from many conventional unfitted mesh methods. The resulting linear algebraic system remains symmetric and positive definite unlike the one obtained from Petrov-Galerkin formulation~\cite{2020GuoLinZou}.
\LC{Again due to the usual low \RG{piecewise $\bfH^1(\curl)$ regularity} near the interface for the $\bfH(\curl)$ equations, in this work we only consider the lowest order methods.}

The rest of this article is organized as follows. In Section 2, some existing results are presented to help us to establish the error analysis. In Section 3, we introduce the IVE space and its properties, and review IFE spaces. In Section 4, we show some novel estimates for IFE spaces that help in our error analysis. In Section 5 and Section 6, the convergence is shown for the $H^1$ and $\bfH(\curl)$ interface problems, respectively.

\section{Preliminary}
In this section, we introduce some mesh assumptions and define some notation. We also recall some existing fundamental estimates which are critical for our analysis. Throughout this paper, we assume $\Omega\subset \mathbb R^2$ is a simply connected convex polygon. Usually it can be chosen as a rectangle enclosing the interface. 

\subsection{Meshes}
\label{sec:meshes}
Let $\mathcal{T}_h = \{K\}$ be a shape regular triangulation of the domain $\Omega$ that may not be fitted to the interface. A triangle $K$ is called an interface triangle if $|K\cap \Omega^{+}| >0$ and $|K\cap \Omega^{-}|>0$; otherwise $K$ is called a non-interface element. The collection of interface elements and non-interface elements are denoted as $\mathcal{T}^i_h$ and $\mathcal{T}^n_h$, respectively. 

For a non-interface element $K$, the local finite element space is simply defined as the linear polynomial space $\mathbb{P}_1(K)$ for \eqref{eq:problem-pde-interface} or the lowest order N\'ed\'elec space $\mathcal{ND}_0(K)$~\cite{Nedelec1980,2003Monk} for \eqref{inter_PDE}. The usage whether to choose the nodal or edge shape functions depends on the problem. 
For convenience of the reader, $\mathcal{RT}_0(K)$ is the lowest order Raviart-Thomas space~\cite{Raviart.P;Thomas.J1977} on $K$ as well. If $K\in \mathcal{T}^i_h$, see Figure \ref{fig:interface-single} for example, $\mathbf{b}_1$ and $\mathbf{b}_2$ denote the intersection points of the interface and $\partial K$, and we let $\Gamma^K_h={\mathbf{b}_1\mathbf{b}_2}$. In addition, we let $\mathcal{N}_K$ be collection of vertices and cutting points of $K$, and let $\mathcal{E}_K$ be collection of cut segments from the original edges of $K$, for example $\mathcal{N}_K=\{\mathbf{a}_1,\mathbf{a}_2,\mathbf{a}_3,\mathbf{b}_1,\mathbf{b}_2\}$ and $\mathcal{E}_K=\{\mathbf{a}_1\mathbf{b}_1,\mathbf{b}_1\mathbf{a}_2,\mathbf{a}_2\mathbf{a}_3,\mathbf{a}_3\mathbf{b}_2,\mathbf{b}_2\mathbf{a}_1\}$ for the interface element $K$ in Figure \ref{fig:interface-single}. Namely, we treat $K$ as pentagon instead of a triangle. Moreover, let $\mathcal{N}_h$ and $\mathcal{E}_h$ be the collection of all the vertices and edges of $\mathcal{N}_K$ and $\mathcal{E}_K$ overall all the $K$, respectively. \RG{Although the conventional IFE methods may be only used on the element in Figure \ref{fig:interface-single} that has two cutting point on two different edges, the proposed method can be readily used for elements with more complex geometry such as those in Figures \ref{fig:interface-multi-2} and \ref{fig:interface-multi}. %
}

\begin{figure}[htbp]
\begin{center}
\begin{subfigure}[b]{0.3\linewidth}
      \centering
      \includegraphics[width=0.8\textwidth]{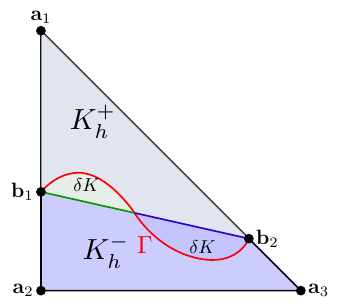}
      \caption{}
      \label{fig:interface-single}
\end{subfigure}
\begin{subfigure}[b]{0.3\linewidth}
      \centering
      \includegraphics[width=0.8\textwidth]{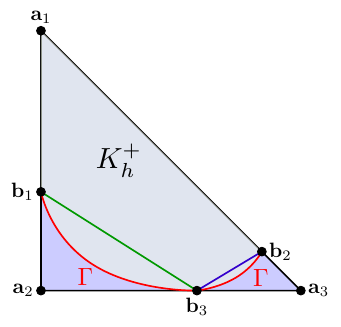}
      \caption{}
      \label{fig:interface-multi-2}
\end{subfigure}
\begin{subfigure}[b]{0.3\linewidth}
  \centering
  \includegraphics[width=0.8\textwidth]{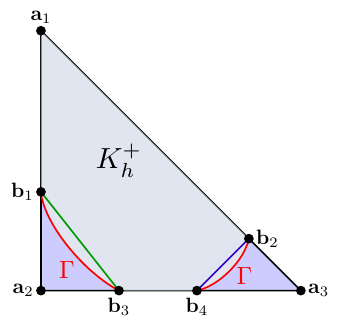}
  \caption{}
  \label{fig:interface-multi}
\end{subfigure}%
\end{center}
\caption{Possible configuration for an interface element.  \eqref{fig:interface-single}, \eqref{fig:interface-multi-2}, \eqref{fig:interface-multi}: $\Gamma$ intersects at an interface at 2, 3, 4 points. \RG{The proposed IVE spaces can be defined on almost arbitrary interface element configuration, as discussed in Section \ref{subsec:IVE}. But the construction of IFE spaces and the error analysis will be a little more technical for those general cases. So for simplicity, we will only consider the case in (2a) for the discussion starting from Section \ref{subsec:IFE}. } }
\label{fig:interface}
\end{figure}

We define the union of cut segments $\Gamma^K_h$ of all the interface elements as the approximated interface $\Gamma_h$, which also separates the original domain $\Omega$ into two subdomains $\Omega_h^{\pm}$, in which the $\pm$ are determined by the area overlap with $\Omega^{\pm}$. Define $\alpha_h=\alpha^{\pm}, \beta_h=\beta^{\pm}$ on $\Omega_h^{\pm}$. For each interface triangle $K$, $\delta{K}$ is the subset of $K$ such that $\beta\neq \beta_h$ (i.e. mismatch region). Using Figure \ref{fig:interface-single} as an example, without loss of generality, $K^{+}_h:=\LC{\rm int}\operatorname{Conv} (\mathbf{a}_1\mathbf{b}_1\mathbf{b}_2)$ and $K^{-}_h$ is the quadrilateral complement formed by $\LC{\rm int}\operatorname{Conv} (\mathbf{a}_2\mathbf{a}_3\mathbf{b}_2\mathbf{b}_1)$, \LC{where {\rm int} stands for the interior so that $K^{\pm}_h$ are open sets}, and the relevant definitions and proofs follow similarly when $\pm$ swaps.

\subsection{Sobolev Spaces and Norms}
\label{subsec:SpaceNorms}

Let $H^k(D)$ ($k \ge 0$) be the standard Sobolev space on a domain $D$ with the norm $\|\cdot\|_{H^k(D)}$, as well as the seminorm $|\cdot|_{H^k(D)}$ when $k>0$. Due to the discontinuity of the coefficient $\beta$, the solution to the $H^1$ interface problem in
\eqref{eq:problem-pde-interface} is not in $H^2(\Omega)$ globally. To define this piecewise Sobolev space, for any \LC{open} subdomain $D\subset\Omega$ intersecting $\Gamma$, \SC{$D^{\pm} := D\cap \Omega^{\pm}$}, we introduce
\[
H^k(\cup D^{\pm})=\bigl\{u\in H^1(D) \text{ and } u^{\pm}\in 
H^k(D^{\pm})\bigr\}
\] 
and the piecewise $H^k$--norm is defined by $\|u\|^2_{H^k(\cup D^{\pm})}=\|u\|^2_{H^k(D^{+})}+\|u\|^2_{H^k(D^{-})}$ for any $u\in H^k(\cup D^{\pm})$. If there is no danger of confusion, in the following discussion, we shall employ a simple notation for the norms: $\|\cdot\|_{k,D}=\|\cdot\|_{H^k(D)}$ and $\|\cdot\|_{k,\cup D^{\pm}}=\|\cdot\|_{H^k(\cup D^{\pm})}$, and the seminorms follow similarly.  
For the $\bfH(\curl)$ interface problem, we let
\[
\begin{aligned}
\bfH^k(\curl;D) &=  \{ \bfu \in \bfH(\curl;D): \curl\,\bfu \in \bfH^k(D) \}, \\
 \bfH^k(\div;D) &= \{ \bfu \in \bfH(\div; D): \div\,\bfu \in \RG{H^k(D)} \}.
\end{aligned}
\]
In addition, we introduce the following spaces
\begin{subequations}
\label{tildeHspace}
\begin{align}
      H^2(\beta;\mathcal T_h) = &\, H^1(\Omega)  \cap \{u:~ u|_K\in H^2( K), \, \forall K\in\mathcal{T}^n_h \} \cap \label{tildeHspace1} \\
     & \quad \{ u:~ u|_K\in H^2(\cup\, K^{\pm}),\; \beta\nabla u|_K \in \bfH(\div;K),  \, \forall K\in\mathcal{T}^i_h \},  \notag 
\\
      \bfH^1(\curl,\alpha,\beta;\mathcal T_h)  = &\, \bfH(\curl;\Omega) \cap  \{ \bfu:~ \bfu|_K\in \bfH^1(\curl; K),  \, \forall K\in\mathcal{T}^n_h \} \cap \label{tildeHspace2} 
\\
     & \quad \{ \bfu:~ \bfu|_K\in \bfH^1(\curl;\cup\, K^{\pm}) ,\; \beta \bfu|_K \in \bfH(\div;K),\; \alpha\curl\,\bfu|_K \in H^1(K), \, \forall K\in\mathcal{T}^i_h  \}, \notag    
\\
      \bfH^1(\div,\beta;\mathcal T_h)  = &\, \bfH(\div;\Omega) \cap \{\bfu:~ \bfu|_K \in \bfH^1(\div; K), \, \forall K\in\mathcal{T}^n_h \} \cap \label{tildeHspace3}
\\ 
     & \quad \{ \bfu:~ \bfu|_K\in \bfH^1(\div;\cup\, K^{\pm}),\; \beta \bfu|_K \in \bfH(\curl;K),   \, \forall K\in\mathcal{T}^i_h  \},   \notag 
\\
      H^{1}(\alpha; \mathcal T_h) = &\, L^2(\Omega)\cap \{u:~ u|_K\in H^1(K), \, \forall K\in\mathcal{T}^n_h \} \cap\label{tildeHspace4} \\
     & \quad \{u:~ u|_K\in H^1(\cup\, K^{\pm}),\, \alpha u|_K\in H^1(K), \, \forall K\in\mathcal{T}^i_h \}.  \notag
\end{align}
\end{subequations}
It is not hard to see these spaces are mesh-dependent and are constructed based on the associated jump conditions. Under the setting introduced in Section \ref{sec:intro} that $f\in L^2(\Omega)$ and $\Gamma\in C^{1,1}$, it can be shown that (see e.g.~\cite{1999ChenZou,2010ChuGrahamHou,2002HuangZou}), the solution to the $H^1$ elliptic interface problem satisfies $u\in H^2(\cup\Omega^{\pm})$, and thus, with the jump conditions, $u\in H^2(\beta;\mathcal T_h)$. As for the $\bfH(\curl)$ interface problem, we follow \cite{Huang;Zou:2007Uniform,2012HiptmairLiZou} to assume $\bfu\in \bfH^1(\curl;\cup\, \Omega^{\pm})$. With the jump condition, we have $\bfu\in \bfH^1(\curl,\alpha,\beta;\mathcal T_h)$.

Given an interface element $K$, we let $H^{2}(\beta;K)$, $\bfH^1(\text{curl},\alpha,\beta;K)$ and $H^{1}(\alpha;K)$ be the local spaces on $K$ of their respective global counterpart in \eqref{tildeHspace}, with the inter-element continuity constraint removed. These spaces together with the classic Sobolev spaces admit the following diagram in the continuous level:
\begin{equation}
\label{DR_curl_continu}
\left.\begin{array}{ccccccc}
\mathbb{R} \xrightarrow[]{\quad} & H^{2}(\beta;K) & \xrightarrow[]{~~\nabla~~} & \bfH^1(\text{curl},\alpha,\beta;K)  & \xrightarrow[]{~~\curl~~} & H^{1}(\alpha;K) & \xrightarrow[]{\quad}  0\\
~~~~& \quad \bigg\downarrow I &  & ~~~~\bigg\downarrow \beta  &  &~~~~\bigg\downarrow \alpha \\
0 \xleftarrow[]{\quad}& L^2(K) & \xleftarrow[]{~~~ \div ~~~} & \bfH(\div;K) & \xleftarrow[]{~~~\bcurl~~~} & H^1(K)  &\xleftarrow[]{\quad} \mathbb{R}.
\end{array}\right.
\end{equation}
We highlight that the scalar multiplication $\beta\cdot$ and $\alpha\cdot$ can be understood as Hodge stars \cite{2000DouglasRichardRagnar} as shown by the downward arrows in \eqref{DR_curl_continu}. We shall construct virtual element spaces to mimic this diagram in the discrete level.
\begin{lemma}
\label{lem_exact_seq_contin}
\RG{Assume that $\Gamma$ is $C^2$ smooth and $\partial \Omega$ is a polygon}, $\Gamma\cap \partial \Omega = \varnothing$. Then, 
\begin{equation}
\label{lem_exact_seq_contin_eq1}
\left.\begin{array}{ccccccc}
\mathbb R \xrightarrow[]{\quad} & H^2(\beta;\mathcal T_h) & \xrightarrow[]{~~\nabla ~~} & \bfH^1(\curl,\alpha,\beta;\mathcal T_h) & \xrightarrow[]{~~\curl~~} & H^1(\alpha;\mathcal T_h)  &\xrightarrow[]{\quad} 0
 \end{array}\right.
\end{equation}
is exact.
\end{lemma}
\begin{proof}
We first recall the standard exact sequence of the de Rham complex:
\begin{equation}
\left.\begin{array}{ccccc}
\mathbb{R} \xrightarrow[]{\quad} H^{1}(D)  \xrightarrow[]{~~\nabla~~}  \bfH(\text{curl};D)  \xrightarrow[]{~~\curl~~}  L^2(D)\xrightarrow[]{\quad}  0,
\end{array}\right.
\end{equation}
where $D$ is any contractible subdomain of $\Omega$ with Lipschitz boundary $\partial D$. 

By definition, for $v\in H^{2}(\beta;\mathcal{T}_h)$, $\nabla v\in \bfH(\curl;\Omega)$ satisfies the regularity condition and the jump conditions associated with $\bfH^1(\curl,\alpha,\beta;\mathcal T_h)$ and obviously $\curl \nabla v = 0$. Conversely, let $\bfu\in \bfH^1(\curl,\alpha,\beta;\mathcal T_h)$ such that $\curl\,\bfu=0$. We are going to find $v\in H^{2}(\beta;\mathcal{T}_h)$ such that $\bfu = \nabla v$. By the standard exact sequence, there exists $v\in H^1(\Omega)$ such that $\nabla v  = \bfu$. We need to verify the extra conditions associated with $H^2(\beta;\mathcal T_h)$ for $v$. Given each $K\in\mathcal{T}^n_h$, $\nabla v = \bfu \in \bfH^1(K)$ implies $v\in H^2(K)$. On each $K\in\mathcal{T}^i_h$, similarly $\nabla v = \bfu \in \bfH^1(\cup K^{\pm})$ implies $v\in H^2(\cup K^{\pm})$. In addition, $\beta \nabla v\in \bfH(\div;K)$ is trivial by \eqref{tildeHspace2}. Thus, $v \in H^2(\beta;\mathcal T_h)$.

Next, let us show $\curl: \bfH^1(\curl,\alpha,\beta;\mathcal T_h)  \to H^{1}(\alpha;\mathcal T_h)$ is surjective. \RG{Since $\Omega$ may not be convex, we let $\widetilde{\Omega}$ be convex hull of $\Omega$. For any $f\in H^{1}(\alpha;\mathcal T_h) \subset L^2(\Omega)$, we let $\tilde{f}$ be the trivial zero extension of $f$ to $\widetilde{\Omega}$, and thus $\tilde{f}\in L^2(\widetilde{\Omega})$. Since $\Gamma$ does not intersect $\partial\Omega$, it can also partition $\widetilde{\Omega}$ into interior and exterior subdomains denoted by $\widetilde{\Omega}^{\pm}$. Then, $\beta^+$ can be naturally used on $\widetilde{\Omega}^+$. Thus, without loss of generality, we shall keep the same notation.} Consider a function $\varphi$ such that
\begin{subequations}
\label{lem_exact_seq_contin_eq2}
\begin{align}
\curl\, \beta^{-1} \bcurl \, \varphi &=  -\div( \beta^{-1} \nabla \varphi) =  \tilde{f} & \text{in} \;  \cup\, \widetilde{\Omega}^{\pm},  \label{lem_exact_seq_contin_eq21}
\\
[\varphi]_{\Gamma} &= 0, &\text{on} \; \Gamma, \label{lem_exact_seq_contin_eq22} \\
[ \beta^{-1}\nabla \varphi\cdot\bfn]_{\Gamma} &= 0, &\text{on} \; \Gamma, \label{lem_exact_seq_contin_eq23} \\
\beta^{-1}\nabla \varphi\cdot\bfn & = |\partial \widetilde{\Omega}|^{-1} \int_{\Omega}  \tilde{f} \dd \bfx &\text{on} \;  \partial \widetilde{\Omega}.
\end{align}
\end{subequations}
Note that \eqref{lem_exact_seq_contin_eq2} is a pure Neumann boundary value problem with the compatibility satisfied, which guarantees the solution $\varphi$ being unique up to a constant. 
Thus, $\bcurl \, \varphi$ is unique and we let $\tilde{\bfw} = \beta^{-1} \bcurl \, \varphi$. As $\tilde{f} \in L^2(\widetilde{\Omega})$ with $\widetilde{\Omega}$ being convex and $\Gamma$ being $C^2$ smooth, by the elliptic regularity we have $\varphi \in H^2(\cup \widetilde{\Omega}^{\pm})$ \cite{2002HuangZou}, thus further obtain $\tilde{\bfw}|_{\widetilde{\Omega}^{\pm}} \in \bfH^1(\widetilde{\Omega}^{\pm})$. Besides, \eqref{lem_exact_seq_contin_eq21} shows $\curl\, \tilde{\bfw} = \tilde{f}$ in $\widetilde{\Omega}^{\pm}$, \eqref{lem_exact_seq_contin_eq22} shows $[\beta\tilde{\bfw}\cdot\bfn]_{\Gamma} = [\bcurl\,\varphi\cdot\bfn]_{\Gamma} = [\nabla\,\varphi\cdot\bft]_{\Gamma} = 0$, and \eqref{lem_exact_seq_contin_eq23} yields $[\tilde{\bfw}\cdot\bft]_{\Gamma} = [\beta^{-1}\bcurl\,\varphi\cdot\bft]_{\Gamma} = [\beta^{-1}\nabla\,\varphi\cdot\bfn]_{\Gamma}=0$. Therefore, $\bfw: = \tilde{\bfw}|_{\Omega} \in \bfH^1(\curl,\alpha,\beta;\mathcal T_h)$ which completes the proof.
\end{proof}
\newrevision{\begin{remark}
Lemma \ref{lem_exact_seq_contin} heavily relies on the smoothness of $\Gamma$ due to the interface problem \eqref{lem_exact_seq_contin_eq2}. If the interface is non-smooth or touches the boundary $\partial\Omega$, the solution regularity, in general, will not be as high as $\bfH^1(\curl)$. 
\end{remark}
}

Next, $u_E^{\pm} := Eu^{\pm}\in H^2(\Omega)$ denotes the standard smooth Sobolev extensions that are bounded in the $H^2$-norm (see e.g.,~\cite{adams2003sobolev}). As for the $\bfH(\curl)$ spaces, the continuous extension operator is given by the following result:
\begin{theorem}[Theorem 3.4 and Corollary 3.5 in~\cite{2012HiptmairLiZou}]
\label{thm_ext}
There exist two bounded linear operators
\begin{equation}
\label{thm_ext_eq0}
\bfE^{\pm}_{\emph{curl}} ~:~ \bfH^1(\emph{curl};\Omega^{\pm})\rightarrow \bfH^1(\emph{curl};\Omega)
\end{equation}
such that for each $\bfu\in\bfH^1(\emph{curl};\Omega^{\pm})$:
\begin{itemize}
  \item[1.]  $\bfE^{\pm}_{\emph{curl}}\bfu = \bfu ~~ \text{a.e.} ~\text{in} ~ \Omega^{\pm}$.
  \smallskip
  \item[2.] $\| \bfE^{\pm}_{\emph{curl}}\bfu \|_{\bfH^1(\emph{curl};\Omega)} \le C_E \| \bfu \|_{\bfH^1(\emph{curl};\Omega^{\pm})}$ with the constant $C_E$ only depending on $\Omega$ and $\Gamma$.
\end{itemize}
\end{theorem}
Using the extension operators, we can define $\bfu^{\pm}_E = \bfE^{\pm}_{\text{curl}}\bfu^{\pm}$ which are the keys in the analysis later.

In the rest of the paper, all constants in $\lesssim$ are $\beta$ and $\alpha$--dependent but independent of the cut point locations unless stated otherwise.

\subsection{Fundamental Inequalities}
We review some fundamental estimates that are crucial for our analysis. The first one concerns the mismatch region of the partitions by the exact interface $\Gamma$ and $\Gamma^K_h$, i.e., $\delta{K}$. For any subdomain $D\subseteq\Omega$ with the interface $\Gamma$, define
$$
D_{\delta} = \{x\in D: {\rm dist}(x,\Gamma)< \delta \}.
$$
Clearly, there hold
$$
\medcup_{K\in\mathcal{T}^i_h} \delta K  \subset \Omega_{\delta_0}, \qquad \text{and} \qquad \medcup_{K\in\mathcal{T}^i_h} K \subset \Omega_{h_\Gamma}
$$
where $\delta_0$ is the maximum distance from $\Gamma_h$ to $\Gamma$, while $h_{\Gamma}$ is the maximum diameter of $K\in\mathcal{T}^i_h$ with $h_{\Gamma}\lesssim h$. By well-known geometric estimates, e.g., see~\cite[Lemma 3.2]{Guo;Lin:2019immersed}, we have $\delta_0\lesssim h^2$. The following result can be found in~\cite{chen2015adaptive,2010LiMelenkWohlmuthZou}.

\begin{lemma}[A norm estimate on a strip region]
\label{strip region}
For each $u\in H^1(\cup\, \Omega^{\pm})$, there holds 
\[
\|u\|_{0,\Omega_{\delta}}\lesssim \sqrt{\delta}\|u\|_{1,\cup\, \Omega^{\pm}}.
\]
\end{lemma}

We will also need the following trace theorems and Poincar\'e-type inequalities.

\begin{lemma}[A trace inequality~\cite{Brenner;Sung:2018Virtual}]
\label{H1 trace}
Let $e$ be an edge of a shape regular element $K$. Then, for all $v\in H^1(K)$, there holds
\[
\|v\|_{0, e}^2\lesssim h^{-1} \|v\|_{0,K}^2+h|v|_{1,K}^2.
\]
\end{lemma}

\begin{lemma}[A trace inequality on interface \cite{2016WangXiaoXu}]
\label{trace_surf}
On any interface element $K$, for all $v\in H^1(K)$, there holds
\begin{equation}
\label{trace_surf_eq0}
\| v \|_{0, \Gamma^K_h} \lesssim h^{-1/2}_K \| v \|_{0,K} + h^{1/2}_K | v |_{1,K}.
\end{equation}
\end{lemma}

\begin{theorem}[\RG{Poincar\'e-Friedrichs}' type inequalities \RG{\cite[Lemma 6.8]{2018CaoChen} and \cite[(2.14)]{Brenner;Sung:2018Virtual} }]
\label{Poincare} 
Given a polygon $K$ with Lipschitz boundary $\partial K$ and the number of edges of $K$ is uniformly bounded, \LC{for $v\in C^0(\partial K)$ and piecewise linear on $\partial K$}, there holds, for each $e\subset \partial K$,
\begin{subequations}
\label{Poincare_eq}
\begin{align}
    & \|v\|_{0, e}\lesssim h_K^{-1/2}\snorm{ \int_{\partial K} v \dd s } + h_K^{1/2}|v|_{1/2,\mathcal{E}_K}, ~~~ \label{eq:Poincare-boundary}
\end{align}
wherein the seminorm $|\cdot|_{1/2,\mathcal{E}_K}$ \RG{is defined in \eqref{half_semi_norm_2}}.
Furthermore, if $K$ is shape regular in the sense that it is star-shaped with respect to a disk with radius $\rho h_K$, then for each $v\in H^1(K)$, there hold
\begin{align}
    &  \|v\|_{0, K}\lesssim \snorm{ \int_{\partial K} v \dd s } + h_K |v|_{1,K}. ~~~  \label{Poincare_eq2} 
\end{align}
\end{subequations}

\end{theorem}

\section{Immersed Virtual Element and Immersed Finite Element Spaces}

In this section, we introduce the immersed virtual element (IVE) and the immersed finite element (IFE) spaces. Then we describe the associated projection and interpolation operators. \RG{We connect} them by commuting diagrams.

\subsection{Immersed Virtual Element Spaces}
\label{subsec:IVE}
The proposed IVE space is a group of novel virtual element spaces with an interface immersed inside the element. \RG{As only the linear method is considered in this article, the interface is flattened inside each element, i.e., the whole interface is approximated by a polyline $\Gamma_h$. We let $\alpha_h$ and $\beta_h$ be the piecewise constant coefficients with interface being $\Gamma_h$. As we only consider the lowest order methods, such a linear approximation to the geometry is sufficient.}

\subsubsection{$H^1$ Virtual Element Spaces}
For each interface element $K$, we begin with an $H^1$ virtual element space that encodes the interface into its elements:
\begin{equation}
\begin{split}
\label{virtual_space}
V^n_h(K) = \{ v_h:~ &\mathrm{div} (\beta_h \nabla v_h) = 0, ~~ v_h|_e\in\mathbb{P}_1(e), ~ \forall e\in\mathcal{E}_K, ~~ v_h|_{\partial K}\in C^0(\partial K), \\
& [v_h]_{\Gamma^K_h} = 0, \text{ and } [\beta_h\nabla v_h\cdot \bar{\bfn}]_{\Gamma^K_h} = 0  \}.
 \end{split}
\end{equation}
Here we note that the jump conditions in \eqref{virtual_space} are imposed on the approximated interface $\Gamma^K_h$ instead of on the exact interface $\Gamma$, the barred notation $\bar{\bfn}$ denotes the unit normal vector to $\Gamma^K_h$ that points roughly in the same direction with $\bfn$. Similarly, $\bar{\bft}$ is the unit tangential vector to $\Gamma^K_h$ that is an approximation to $\bft$. The motivation to impose the jump conditions on $\Gamma^K_h$ is that the IFE space defined later becomes a subspace of $V^n_h(K)$, which facilitates a simpler analysis. There will be no essential difficulty if the jump conditions of the virtual element spaces are defined on $\Gamma$ as the analysis follows the VEM meta-framework.

Clearly, $V^n_h(K)$ is not empty. The reason is that we can treat $V^n_h(K)\subseteq H^1(K)$ as the space of the weak solutions to a boundary value problem. Then, the dimension of $V^n_h(K)$ is that of the boundary conditions, i.e., the dimension of $\Pi_{e\in \mathcal{E}_K} \mathbb P_1(e) \cap C^0(\partial K)$, which further can be identified by the number of the vertices on $\partial K$. 
Consequently, $V^n_h(K)$ is unisolvent: if the DoFs $v_h(\bfx)=0$ at each $\bfx \in \mathcal{N}_K$, which implies the boundary value $v_h=0$ on $\partial K$, then $v_h\equiv 0$ in $K$ by the uniqueness of the local problem. This space can be understood as a natural generalization of the classic linear virtual element space in~\cite{2013BeiraodeVeigaBrezziCangiani,2014VeigaBrezziMariniRusso} to the case of discontinuous coefficients. Furthermore, referred to \eqref{1D_loc}, we can see the space is also a generalization of 1D space by Babu\v{s}ka et al. in \cite{1994BabuskaCalozOsborn,1983BabuskaOsborn}.

Note that 
$
V^n_h(K) \subset H^1(K)\cap \{\beta_h \nabla u\in \bfH(\div; K)\}.
$
Inside the interface element $K$, the piecewise constant function $\beta_h$ serves as a Hodge star which maps the function $\nabla u\in \bfH(\curl;K)$ to a function $\beta_h\nabla u\in \bfH(\div; K)$. 

The global space is then defined as
\begin{equation}
\label{virtual_space_glob}
V^n_h= \{ v_h\in H^1_0(\Omega): v|_K\in V^n_h(K)~ \text{if} ~ K\in\mathcal{T}^i_h ~ \text{and} ~ v|_K\in \mathbb{P}_1(K)~ \text{if} ~ K\in\mathcal{T}^n_h \}
\end{equation}
which is an $H^1$-conforming space. Lastly, we can define the Lagrange type interpolation $I^n_h$ using the nodal DoFs, for continuous $u$, 
\begin{equation}
\label{vem_Lag_interp}
(I^n_h u)(\bfx) = u(\bfx), \quad\quad\quad \forall \bfx \in \mathcal{N}_h.
\end{equation}

\subsubsection{$\bfH(\curl)$ Virtual Element Spaces}
\label{sec:space-hcurl}
Next, let us consider an $\bfH(\curl)$ virtual element space involving discontinuous coefficients. Given an interface element $K$, we define
\begin{align}
\label{virtual_space_hcurl_loc}
\bfV^e_h(K) = \{ \bfv_h \in \; \bfH(\curl ;K) : ~ & \beta_h\bfv_h \in\bfH(\text{div};K),  ~\nonumber \bfv_h\cdot\bft_e|_e \in \mathbb{P}_0(e), ~\forall e \in \mathcal{E}_K,\\
 &\text{div}\,\beta_h \bfv_h = 0, ~ \alpha_h\curl \, \bfv_h \in \mathbb{P}_0(K) \}.
\end{align}
Again $\beta_h$ is a Hodge star operator which maps $\bfv_h \in \; \bfH(\curl ;K)$ to $\beta_h\bfv_h \in\bfH(\text{div};K)$ and $\alpha_h$ is another Hodge star which maps $\curl \, \bfv_h\in L^2(K)$ to $\alpha_h\curl \, \bfv_h\in H^1(K)$. 

With this definition, it is easy to see
\begin{equation}
\label{curlV_space}
\curl \, \bfV^e_h(K) = \{ c  ~ \text{is a piecewise constant on } \, K^{\pm}_h:~ \alpha_h^+c^+ = \alpha_h^-c^- \}. %
\end{equation}
In the rest of this section, we denote the weighted average of $\alpha$ on $K$ by
\begin{equation}\label{alphaK}
\alpha_K = \big(|K^+_h|\alpha_h^- + |K^+_h|\alpha_h^+\big)/|K|.
\end{equation}
If a piecewise constant vector $\bfc:=\bfc^{\pm}$ on $K^{\pm}_h$ satisfies $\beta_h^+\bfc^+\cdot\bar{\bfn}= \beta_h^-\bfc^-\cdot\bar{\bfn}$ and $\bfc^+\cdot\bar{\bft}= \bfc^-\cdot\bar{\bft}$, 
then $\bfc \in \bfV^e_h(K)$. Thus, $\bfV^e_h(K)$ is non-empty, and upon a closer inspection it is not hard to see that the aforementioned $\bfc^{\pm}$ form the gradient of the $H^1$ IFE space on $K$ (see the definition of IFE spaces in Subsection \ref{subsec:H1_IFE}).
\LC{The dimension of $\bfV^e_h(K)$ is not immediately obvious.}
\LC{To characterize the functions in $\bfV^e_h(K)$, we consider the following local problem}: given $f\in L^2(K)$ and $g\in L^2(\partial K)$, find $\bfv_h$ such that
\begin{subequations}
\label{eq:pb-hcurl-loc}
\begin{align}
\mathrm{curl}~\bfv_h = f \quad & \text{in} \;  K, 
\\
\mathrm{div}(\beta_h \bfv_h) = 0 \quad & \text{in} \; K, 
\\
\bfv_h\cdot\bft = g \quad & \text{on} \; \partial K.
\end{align}
\end{subequations}
The following lemma establishes the well-posedness of this equation.
\begin{lemma}
\label{lem_hcurl_wellposed}
The equation in \eqref{eq:pb-hcurl-loc} is well-posed if the compatibility condition is met:
\begin{equation}
\label{eq:hcurl-loc-compatibility} 
\int_K f \dd \bfx = \int_{\partial K} g \dd s.
\end{equation}
\end{lemma}
\begin{proof}
By the constraint $\text{div}\,\beta_h \bfv_h = 0$, the exact sequence property implies that there exists $\varphi\in H^1(K)$ such that $\bcurl \, \varphi = \beta_h \bfv_h$. Then the argument basically mimics the one in the proof of Lemma \ref{lem_exact_seq_contin} locally on an element. In particular, the problem \eqref{eq:pb-hcurl-loc} then becomes a pure Neumann problem:
\begin{subequations}
\label{eq:pb-hcurl-potential}
\begin{align}
\mathrm{curl} (\beta_h^{-1} \mathbf{curl}\, \varphi) = - \mathrm{div} (\beta_h^{-1} \nabla  \varphi) = f \quad & \text{in} \;  K, 
\\
\beta_h^{-1} \nabla \varphi\cdot\bfn = -g \quad & \text{on} \; \partial K.
\end{align}
\end{subequations}
Clearly, for any boundary condition $g$ and source term $f$ satisfying the compatibility condition \eqref{eq:hcurl-loc-compatibility}, \eqref{eq:pb-hcurl-potential} has a unique solution $\varphi\in H^1(K)/\mathbb{R}$, and thus a unique $\bfv_h = \beta^{-1}_h\bcurl \,\varphi $.
\end{proof}
\LC{We then follow \cite{Chen;Huang:2020Discrete} to introduce the so-called data space 
$$
\mathcal D(K) = \{(f_0, g_0): \alpha_h f_0\in \mathbb P_0(K), g_0|_e\in \mathbb P_0(e), e\in \mathcal E_K, \text{ and } \int_K f_0 \dd \bfx = \int_{\partial K} g_0 \dd s.\}
$$
Despite the fact $f_0$ gives an extra $1$ dimension, the compatible condition reduces this extra dimension, and thus $\dim \mathcal D(K) = |\mathcal E_K|$. Given a function $\bfu\in \bfV^e_h(K)$, $(\curl \bfv_h, \bfv_h\cdot \bft_{\partial K})$ defines a mapping $\mathcal L: \bfV^e_h(K) \to \mathcal D(K)$. On the other hand, given $(f_0,g_0)\in \mathcal D(K)$, the solution $\bfv_h$ to the local problem \eqref{eq:pb-hcurl-loc} is a function in $\bfV^e_h(K)$. The uniqueness of the local problem implies $\mathcal L^{-1}$ is well-defined. Therefore, $\mathcal L$ is one-to-one, and $\dim \bfV^e_h(K) =  \dim \mathcal D(K) =  |\mathcal E_K|$. 
}
Next we show the DoFs on the edges in $\mathcal{E}_K$ suffice to uniquely determine a function in $\bfV^e_h(K)$ as follows. 
\begin{lemma}
\label{lem_iso_mor}
The DoFs $\{\bfv_h\cdot\bft_e, e\in \mathcal{E}_K \}$ are unisolvent on the space $\bfV^e_h(K)$ for any $K\in \mathcal{T}_h$. 
\end{lemma}
\begin{proof}
\LC{First of all, the number of DoFs $\{\bfv_h\cdot\bft_e, e\in \mathcal{E}_K \}$ is $|\mathcal E_K|$ which is equal to the dimension of the space $\bfV^e_h(K)$.}
\LC{Then, a mapping can be defined from DoFs to the data space $\mathcal D(K)$. An obvious choice is $g_0|_e = \bfv_h\cdot\bft_e \in \mathbb P_0(e)$. From the compatibility condition 
\begin{equation}
\label{lem_iso_mor_eq1}
\sum_{e\in \mathcal E_K}|e|  \bfv_h\cdot\bft_e = \int_{\partial K} \bfv_h\cdot\bft \dd s = \int_{K} \curl\, \bfv_h \dd \bfx = |K^+_h| (\curl \bfv_h)^+ + |K^-_h| (\curl \bfv_h)^-,
\end{equation}
where $(\curl \bfv_h)^{\pm}$ are constant restricted to $K^{\pm}_h$, respectively. 
On the other hand, by definition of the space, we have equation
\begin{equation}\label{lem_eq2}
\alpha^+_h (\curl \bfv_h)^+=\alpha^-_h (\curl \bfv_h)^-.
\end{equation}
Then solve \eqref{lem_iso_mor_eq1}-\eqref{lem_eq2} for $(\curl \bfv_h)^{\pm}$, which are two constant scalars, we get
\begin{equation}
\label{curl_val}
(\curl \bfv_h)^+ = \frac{1}{|K|}\frac{\alpha_h^-}{\alpha_K} \int_{\partial K} \bfv_h\cdot\bft \dd s, \quad
(\curl \, \bfv_h)^- = \frac{1}{|K|}\frac{\alpha_h^+}{\alpha_K} \int_{\partial K} \bfv_h\cdot\bft \dd s,
\end{equation}
where the weighted average $\alpha_K$ is given in \eqref{alphaK}. 
That is to say, the second output of this map in $\mathcal{D}(K)$, $f_0 = \curl \bfv_h$ can be expressed by a linear combination of DoFs.}
\LC{Therefore, the unisolvence follows from the uniqueness of the local problem \eqref{eq:pb-hcurl-loc}. More precisely, if $\bfv_h\cdot\bft_e$ vanishes for every $e\in \mathcal{E}_K$, then both $f_0$ and $g_0$ are zero and consequently the solution $\bfv_h$ to \eqref{eq:pb-hcurl-loc} is zero.}
\end{proof}
The importance of this lemma is that we have established the one-to-one correspondance between the local virtual element space, the DoFs, and the data space. Moreover, \RG{from the proof, $\curl\, \bfv_h$ is readily computable for any $\bfv_h\in \bfV^e_h(K)$ through the DoFs using \eqref{curl_val}, which is vital for the implementation.}

Thanks to the edge DoFs, we can also construct a globally $\bfH(\curl)$-conforming space
\begin{equation}
\begin{split}
\label{virtual_space_hcurl_glob}
\bfV^e_h = \{ \bfv_h \in \bfH(\curl;\Omega) \,:\, &\bfv_h|_K\in \bfV^e_h(K) \,\,\, \text{if} \,\,\, K\in \mathcal{T}^i_h \,\, \text{and} \,\,\, \bfv_h|_K\in \mathcal{ND}_0(K) \,\,\, \text{if} \,\,\, K\in\mathcal{T}^n_h \}.
\end{split}
\end{equation}
We can define the edge interpolation $I^e_h \bfu$ as, provided that $\bfu$ is smooth enough,
\begin{equation}
\label{vem_edge_interp}
\int_e I^e_h \bfu \cdot \bft \dd s = \int_e \bfu \cdot \bft \dd s, \quad\quad \forall e\in \mathcal{E}_h.
\end{equation}

\subsubsection{$\bfH(\div)$ Virtual Element Spaces}
Similarly, given an interface element $K$, the $\bfH(\div)$ virtual element space involving discontinuous coefficients is defined as
\begin{align}
\label{virtual_space_hdiv_loc}
\bfV^f_h(K) = \{ \bfv_h \in \; \bfH(\div ;K) : ~ & \beta^{-1}_h\bfv_h \in\bfH(\curl;K),  ~\nonumber \bfv_h\cdot\bfn_e \in \mathbb{P}_0(e), ~\forall e\in \mathcal{E}_K,\\
 & \text{div}\, \bfv_h \in \mathbb{P}_0(K), ~ \curl \,\beta^{-1}_h \bfv_h =0 \}.
\end{align}
Different from the standard virtual element spaces, $\bfV^f_h(K)$ is not exactly the rotation of $\bfV^e_h(K)$. Indeed, the Hodge star is defined by $\beta_h^{-1}$ which maps a function $\bfv_h \in \; \bfH(\div ;K)$ to $\beta^{-1}_h\bfv_h \in\bfH(\curl;K)$. 
 
By a similar argument to Section \ref{sec:space-hcurl} that leads to Lemma \ref{lem_iso_mor}, we can show that the definition \eqref{virtual_space_hdiv_loc} 
yields a well-defined local space with DoFs being $\bfv_h\cdot\bfn_e$, $\forall e\in\mathcal{E}_K$, and thus has the dimension $|\mathcal{E}_K|$.
Similarly, $\div\bfv_h$, $\bfv_h\in\bfV^f_h(K) $ is computable using these DoFs through the integration by parts:
\begin{equation}
\label{vem_hdiv_comput}
|K| \div\bfv_h =\int_K \div\bfv_h \dd \bfx = \int_{\partial K}  \bfv_h\cdot\bfn \dd s.
\end{equation}
Then, the global $\bfH(\div)$ virtual element space is
\begin{equation}
\begin{split}
\label{virtual_space_hdiv_glob}
\bfV^f_h = \{ \bfv_h \in \bfH(\div;\Omega) \,:\, &\bfv_h|_K\in \bfV^f_h(K) \,\,\, \text{if} \,\,\, K\in \mathcal{T}^i_h \,\, \text{and} \,\,\, \bfv_h|_K\in \mathcal{RT}_0(K) \,\,\, \text{if} \,\,\, K\in\mathcal{T}^n_h \}.
\end{split}
\end{equation}
We can also define the edge interpolation $I^f_h \bfu$ as, provided $\bfu$ is smooth enough, 
\begin{equation}
\label{vem_edge_interp_rotate}
\int_e I^f_h \bfu \cdot \bfn \dd s = \int_e \bfu \cdot \bfn \dd s, \quad\quad \forall e\in \mathcal{E}_h.
\end{equation}
We note that both $I^e_h$ and $I^f_h$ are just standard edge interpolation on non-interface elements.

\subsubsection{A discrete de Rham Complex}
\label{subsec:deRham}

The commuting diagram and a discrete de Rham complex also hold for the newly constructed virtual element spaces. Given each element $K$ and a weight function $w\in L^2(K)$ that is piecewise constant $w^{\pm}$ on $K^{\pm}_h$, let $\pi^w_K$ be the projection, defined with the inner product $(w\cdot,\cdot)_{K}$, onto 
\begin{equation}
\label{eq:space-weighted-constant}
Q^w_h(K) = \{ c  ~ \text{is a piecewise constant on } \, K^{\pm}_h:~ w^+c^+ = w^-c^- \}.
\end{equation}
Namely, for $z\in L^2(K)$ there holds
\begin{equation}
\label{wprojection}
(w \, \pi^w_Kz,v)_{K} = (w \, z,v)_{K}, ~~~~ \forall v \in Q^w_h(K).
\end{equation}
In particular, if $K$ is simply a non-interface element or $w=1$, $\pi^w_K$ reduces to the standard $L^2$ projection onto $Q^1_h(K)=\mathbb{P}_0(K)$. If $K$ is an interface element and $w=\alpha_h$, from \eqref{curlV_space} we have $Q^{\alpha_h}_h(K)=\curl \, \bfV^e_h(K)$. 

We first summarize the aforementioned Hodge star operators associated with $V^n_h(K)$ and $\bfV^e_h(K)$ through the following diagram
\begin{equation}
\label{hodge_star_loc}
\left.\begin{array}{ccccccc}
\mathbb{R} \xrightarrow[]{\quad} & V^{n}_h(K) & \xrightarrow[]{~~\nabla~~} & \bfV^e_h(K)  & \xrightarrow[]{~~\curl~~} & Q^{\alpha_h}_h(K) & \xrightarrow[]{\quad}  0\\
~~~~& \quad \bigg\downarrow I &  & ~~~~\bigg\downarrow \beta_h  &  &~~~~\bigg\downarrow \alpha_h \\
0 \xleftarrow[]{\quad}& L^2(K) & \xleftarrow[]{~~~ \div ~~~} & \bfH(\div;K) & \xleftarrow[]{~~~\curl~~~} & H^1(K)  &\xleftarrow[]{\quad} \mathbb{R}.
\end{array}\right.
\end{equation}
We note that this diagram exactly mimics the one in \eqref{DR_curl_continu} which shows the proposed spaces nicely inherit this feature locally on interface elements.

Furthermore, given $w\in L^2(\Omega)$ that is piecewise constant weight function on each element and subelement of interface elements, we let $Q^w_h$ be a piecewise constant space satisfying $Q^w_h|_{K}=Q^w_h(K)$. Let the global projection be $\pi^w_h|_K = \pi^w_K$. Then, we have our diagram in \eqref{intro:DR_curl}. Let us proceed to show its exactness and commutative property.

\begin{lemma}\label{lm:complex}
There holds
 \[
 \nabla V^n_h \subset \bfV^e_h\cap \emph{Ker}(\curl). 
 \]
Consequently, together with $\curl \bfV^e_h = Q^{\alpha_h}_h$, the discrete sequence on the bottom of \eqref{intro:DR_curl} is a complex.
\end{lemma}
\begin{proof}
We first show the local subset result and focus on interface elements $K$, since the argument for non-interface elements is standard. Given $v_h\in V^n_h(K)$, $[v_h]_{\Gamma^K_h}=0$ and $[\beta_h \nabla v_h\cdot\bar{\bfn}]_{\Gamma^K_h}=0$ imply $\nabla v_h\in \bfH(\curl;K)$ and $\beta_h \nabla v_h\in \bfH(\div;K)$, respectively. In addition, we also have $\div(\beta_h \nabla v_h)=0$ by the local problem \eqref{eq:pb-hcurl-loc}. Besides, $v_h|_e\in \mathbb{P}_1(e)$ implies $\nabla v_h|_e\cdot\bft \in \mathbb{P}_0(e)$. Moreover, it is trivial that $\curl(\nabla v_h)=0$ which gives the desired local subset result. It leads to the global one by their DoFs. 
\end{proof}

\begin{lemma}
\label{lem_vem_seq}
The discrete sequence on the bottom of \eqref{intro:DR_curl}, 
is exact. 
\end{lemma}
\begin{proof}
Given $\bfv_h\in \bfV^e_h\cap \text{Ker}(\curl)$, there exists a $\varphi_h\in H^1(\Omega)$ such that $\nabla \varphi_h = \bfv_h$ \RG{by the continuous exact sequence}. We need to show $\varphi_h\in V^n_h$. On non-interface elements $K$, as $\bfv_h$ is a constant vector, there simply holds $\varphi_h\in\mathbb{P}_1(K)$. \RG{On any} $K\in \mathcal{T}^i_h$, as $\div(\beta_h\bfv_h)=0$, we also have $-\nabla\cdot(\beta_h \nabla \varphi_h)=0$. The jump conditions for $\varphi_h$ are thus satisfied due to those of $\bfv_h$. It follows from the DoFs $\bfv_h\cdot\bft = \nabla \varphi_h\cdot \bft\in \mathbb{P}_0(e)$ that $\varphi_h\in \mathbb{P}_1(e)$, $\forall e\in \mathcal{E}_K$. These results lead to $\varphi_h\in V^n_h(K)$. Thus, we have $\varphi_h\in V^n_h$ through their DoFs.

To show $\curl$ is surjective, we construct an auxiliary mesh $\mathcal{T}^A_h$ by simply refining interface elements into several triangles. Given $q_h\in Q^{\alpha_h}_h\subset H^1(\alpha_h;\mathcal T_h)$, it is trivial that $q_h$ can be considered as a piecewise constant function on $\mathcal{T}^A_h$. 
Then, the classic exact sequence yields a curl-conforming N\'ed\'elec element $\tilde{\bfw}_h\in \bfH(\curl;\Omega)$ with $\tilde{\bfw}_h|_K\in \mathcal{ND}_0(K)$, $\forall K\in \mathcal{T}^A_h$ such that $\curl\, \tilde{\bfw}_h = q_h$. We set $\bfw_h = I^e_h \tilde{\bfw}_h\in \bfV^e_h$, that is, their edge moments only have to agree on the edges in $\mathcal{E}_h$, not on the extra interior edges to form $\mathcal{T}^A_h$. It is trivial that $\curl\,\bfw_h = \curl\,\tilde{\bfw}_h = q_h$ on $K\in \mathcal{T}^n_h$. We only need to verify it on interface elements. Given $K\in\mathcal{T}^i_h$, we let $q^{\pm}_h = q_h|_{K^{\pm}}$. With integration by parts, there holds $\int_K \curl\, \tilde{\bfw}_h \dd \bfx = \int_{\partial K} \tilde{\bfw}_h\cdot\bft \dd s$. Then, by $\alpha^+ q^+_h= \alpha^- q^-_h$ we have
\[
q^{\pm}_h = \frac{1}{|K|} \frac{\alpha_h^{\mp}}{\alpha_K} \int_{\partial K}\tilde{ \bfw}_h\cdot\bft \dd s =\frac{1}{|K|} \frac{\alpha_h^{\mp}}{\alpha_K} \int_{\partial K} \bfw_h\cdot\bft \dd s
\]
with $\alpha_K$ defined in \eqref{alphaK}. Using \eqref{curl_val}, we have concluded $\curl\,\bfw_h = q_h$.
\end{proof}

\begin{remark}
The global N\'ed\'elec edge element constructed in the proof above can be understood as a function in the virtual element space developed in \cite{2021CaoChenGuo} with discontinuous coefficients, in which the DoFs associated with the interior edges of an interface element are eliminated by imposing a single constant $\curl$ value.
\end{remark}

Note that for standard FEM on non-interface elements, $\curl I^e_h \bfu$ is the projection of $\curl \bfu$ onto the constant space, which is the well-known commuting property for the de Rham complex. For the new virtual element spaces, the commuting property also holds.
\begin{lemma}
\label{curl_uI_proj}
The diagram in \eqref{intro:DR_curl} is commutative.
\end{lemma}
\begin{proof}
It suffices to establish the result on interface elements. To this end, given one interface element $K$, we shall first show for any $u\in H^2(\beta;\mathcal T_h)$, there holds
\begin{equation}
\label{curl_uI_proj_eq1}
I^e_h \nabla u = \nabla I^n_h u.
\end{equation}
As shown in Lemma \ref{lm:complex}, we have $\nabla I^n_h u \in \bfV^e_h(K)$; so by the unisolvence in Lemma \ref{lem_iso_mor}, to prove \eqref{curl_uI_proj_eq1}, it remains to check their DoFs coincide. Indeed, given each $e\in\mathcal{E}_K$ with the ending points $\bfa_e$ and $\bfb_e$, we have
\[
\int_e \nabla I^n_h u\cdot \bft \dd s = I^n_hu(\bfa_e) - I^n_hu(\bfb_e) = u(\bfa_e) - u(\bfb_e) = \int_e \nabla u \cdot\bft \dd s = \int_e I^e_h \nabla u \cdot\bft \dd s.
\]
Furthermore, we need to show for any $\bfu\in \bfH^1(\text{curl},\alpha,\beta;\mathcal T_h)$, there holds
\begin{equation}
\label{curl_uI_proj_eq2}
\curl \, I^e_h \bfu = \pi^{\alpha_h}_K \curl\, \bfu.
\end{equation}
Note that functions in $Q^{\alpha_h}_h(K)$ are simply $\alpha^{-1}_h c$ with any constant $c$. Then, Green's theorem gives
\begin{equation}
\begin{split}
\label{curl_uI_proj_eq3}
\int_K \alpha_h \curl I^e_h \bfu \, \alpha^{-1}_h c \dd \bfx 
& = \int_K  \curl I^e_h \bfu ~ c \dd \bfx 
= c \int_{\partial K} I^e_h \bfu\cdot\bft \dd s 
\\
= c \int_{\partial K} \bfu\cdot\bft \dd s
& = \int_K  \curl \bfu ~ c \dd \bfx 
= \int_K \alpha_h \curl\bfu\, \alpha^{-1}_h c \dd \bfx,
\end{split}
\end{equation}
which yields the desired result.
\end{proof}

\begin{remark}
\label{rem_commut_gamma}
It is highlighted that the commutative property essentially only depends on the DoFs of the IVE spaces. It makes the definition of IVE spaces quite flexible. For example, the property still holds if the IVE spaces are defined with the original interface $\Gamma$ instead of the approximate interface $\Gamma_h$, if appropriate jump conditions are imposed on $\Gamma$. 
\end{remark}

Here we assume higher regularity for the spaces in the continuous level so that the canonical interpolation operator $I_h^n$ and $I_h^e$ are well-defined. In the rest of this article, \RG{we simply denote these interpolations by 
\begin{equation}
\label{interp_UI}
u_I:=I_h^n u ~~ \text{and} ~~ \bfu_I:= I_h^e \bfu,
\end{equation}
if there is no confusion}. It is possible to follow the approach in~\cite{Ern;Guermond:2017Finite,Schoberl:2001Commuting} to construct quasi-interpolation operators without extra smoothness requirement and establish the commutative property.

\subsection{Immersed Finite Element Spaces}
\label{subsec:IFE}

Similar to the standard VEM~\cite{2013BeiraodeVeigaBrezziCangiani,2014VeigaBrezziMariniRusso}, the basis functions themselves in the virtual element space do not have explicit pointwise values for computation, and this demands projections. Due to jump conditions, the standard polynomial spaces are not appropriate choices onto which the virtual element spaces are projected. As the IFE space consists of piecewise polynomials satisfying the jump conditions on $\Gamma^K_h$, naturally it can be used as a computable space for projecting. To simplify the discussion, starting from this section, we only consider the interface element in Figure \ref{fig:interface-single}; namely, we make the following assumption:
\begin{assumption}[The background mesh being fine enough]
\label{geometric assumptions}
For each interface triangle $K$ in the background mesh, $\Gamma$ intersects with $K$ at most two distinct points on two different edges.
\end{assumption}
We note that this assumption can be satisfied if $\mathcal T_h$ is sufficiently fine~\cite{2010ChuGrahamHou,Guo;Lin:2019immersed} provided that $\Gamma\in C^{1,1}$, i.e., the interface is locally flat enough.
Even if the interface intersects an element multiple times, such as Figures \ref{fig:interface-multi-2} and \ref{fig:interface-multi}, the proposed method is still applicable, since the immersed virtual functions always exist from solving local problems, which is one of the major difference from the conventional IFEM. So, this assumption is merely to simplify the analysis.
Now, let us review three types of IFE spaces including the $H^1$, $\bfH(\curl)$, and $\bfH(\div)$ spaces.

\subsubsection{$H^1$ IFE Spaces}
\label{subsec:H1_IFE}
First, we consider the $H^1$ case, see e.g., \cite{guo2019improved}. Given an interface element $K$, we consider the approximate jump conditions to \eqref{eq:jump} defined on the segment $\Gamma^K_h$: 
\begin{subequations}
\label{eq:jump-1}
\begin{align}
    v^{+}_h & = v^{-}_h & \text{at} \; \Gamma^K_h ,\label{eq:jump-11} 
\\
   \beta^{+}_h\nabla v^{+}_h \cdot {\bar{\mathbf{n}}} & =\beta^{-}_h\nabla v^{-}_h \cdot {\bar{\mathbf{n}}}  &\text{at} \; \Gamma^K_h. \label{eq:jump-2}  
\end{align}
\end{subequations}
Note that \eqref{eq:jump-11} leads to $\nabla v^{+}_h \cdot \bar{\bft}=\nabla v^{-}_h \cdot {\bar{\mathbf{ t}}}$,
which together with \eqref{eq:jump-2} leads to the relation
\begin{equation}
\label{gradient_ife_1}
 \nabla v^+_{h} =  M  \nabla v^-_{h} \quad ,
\end{equation}
where $M$ is an invertible matrix encoded with the jump information
\begin{equation}
\label{gradient_ife_2}
M = 
\left[\begin{array}{cc}
n^2_2 + \rho n^2_1 &  (\rho-1)n_1n_2 \\
(\rho-1)n_1n_2 & n^2_1+\rho n^2_2
\end{array}\right]
\end{equation}
with $\bar{\mathbf{n}}= ( n_1, n_2 )$, $\bar{\mathbf{t}}=( t_1, t_2 ) = ( n_2, - n_1 )$, and $\rho = \beta^-/\beta^+$. Accordingly, we can express the $H^1$ IFE functions explicitly as follows: 
\begin{equation}
\label{IFE_space_h1_represnt}
v_h(\bfx) = 
\begin{cases}
      &  M\bfc\cdot (\bfx-\bfx^m) + c_0 \quad\quad \text{if} ~ \bfx \in K^+_h, \\
      &  \bfc\cdot (\bfx-\bfx^m) + c_0 \quad\quad\,\,\,\,\,\, \text{if} ~ \bfx \in K^-_h, 
\end{cases}
\end{equation} 
where $\bfx^m=(x^m_1,x^m_2)^{\intercal}$ is the mid-point of $\Gamma^K_h$, and $c_0$ and $\bfc$ are scalar- and vector-valued constants that can be viewed as the DoFs for the polynomial space. Now, the $H^1$ local IFE space on $K$ is then defined as
\begin{equation}
\label{IFE_space}
    S^n_h(K) := \{v_h|_{K^{\pm}_h}\in \mathbb{P}_1(K^{\pm}_h):\;
    v_h \text{ satisfies}~ \eqref{eq:jump-1} \}.
\end{equation}
By counting the number of constraints, $\dim S^n_h(K) = 3$. Comparing it with the virtual element space \eqref{virtual_space}, it is straightforward to conclude that $S^n_h(K)\subset V^n_h(K)$. 
In the classical definition of IFE, e.g.~\cite{Guo;Lin:2019immersed}, the IFE space admits the DoFs as the values at the vertices of a triangular element $K$. In contrast, this nodal basis--DoF pair is now different, as the DoFs are imposed through the virtual element space. Note that the IFE basis in \eqref{IFE_space_h1_represnt} is not the conventional nodal IFE basis, and the formula in \eqref{IFE_space_h1_represnt} is easier to be derived and only used for computing projections.

\subsubsection{$\bfH(\curl)$ IFE Spaces}
The $\bfH(\curl)$ IFE space is developed in~\cite{2020GuoLinZou} which employs the approximate jump conditions for piecewise polynomials $\bfv_h^{\pm}\in\mathcal{ND}_0(K^{\pm}_h)$
\begin{subequations}
\label{weak_jc}
\begin{align}
     \bfv^+_h \cdot\bar{\bft} &=  \bfv^-_h \cdot\bar{\bft}       &\text{at} ~ \Gamma^K_h, \label{weak_jc_1}  \\
     \alpha^+_h \curl~ \bfv^+_h & =  \alpha^-_h \curl~\bfv^-_h          &\text{at} ~ \Gamma^K_h,  \label{weak_jc_2} \\
     \beta^+_h \bfv^+_h\cdot\bar{\bfn} &=  \beta^-_h \bfv^-_h \cdot\bar{\bfn}        &\text{at} ~ \bfx^m. \label{weak_jc_3}
\end{align}
\end{subequations}
Then, the $\bfH(\text{curl})$ IFE space is defined as
\begin{equation}
\label{IFE_loc_spa}
\bfS^e_h(K) = \{  \bfv_h|_{K^{\pm}_h} \in \mathcal{ND}_0(K^{\pm}_h)~: ~ \bfv_h ~ \text{satisfies} ~ \eqref{weak_jc} \}.
\end{equation}

The functions in $\bfS^e_h(K)$ admit the following explicit representation:
\begin{equation}
\label{IFE_curl_represent}
\bfv_h = 
\begin{cases}
      &M\bfc + \frac{c_0}{\alpha^+} (-(x_2 - x^m_2), x_1-x^m_1)^{\intercal} \,\,\,\,\, \text{in} \, K^+_h, \\
      & \,\,\,\,\,\, \bfc + \frac{c_0}{\alpha^-} (-(x_2 - x^m_2), x_1-x^m_1)^{\intercal} \,\,\,\,\, \text{in} \, K^-_h,
\end{cases}
\end{equation}
where $M$ is given by \eqref{gradient_ife_2}, and $c_0$ and $\bfc$ are arbitrary scalar- and vector-valued constants.

\subsubsection{$\bfH(\div)$ IFE Spaces}
To derive a systematic framework, we also recall the $\bfH(\div)$ IFE space~\cite{2021Ji} which is used to approximate $\beta \nabla u\in \bfH(\div;K)$. The related approximate jump conditions are defined as
\begin{subequations}
\label{weak_jcdiv}
\begin{align}
     \bfv^+_h \cdot\bar{\bfn} &=  \bfv^-_h \cdot\bar{\bfn}       &\text{on} ~ \Gamma^K_h, \label{weak_jcdiv_1}  \\
     (\beta^+_h)^{-1} \bfv^+_h \cdot\bar{\bft} &=  (\beta^-_h)^{-1} \bfv^-_h  \cdot\bar{\bft},       &\text{at} ~ \bfx^m. \label{weak_jcdiv_3}
\end{align}
together with the condition
\begin{align}
     \div\, \bfv^+_h &=  \div\, \bfv^-_h.  \label{weak_jcdiv_4}
\end{align}
\end{subequations}
We note that \eqref{weak_jcdiv_4} is proposed in \cite{2021Ji} for guaranteeing unisolvence, but it is interesting to note that it also mimics the condition of the face IVE space in \eqref{virtual_space_hdiv_loc}, i.e. $\div v_h$ is a single constant in $K$.
We emphasize again that the jump condition is from the discrete Hodge star $\beta_h^{-1}$ which maps $\bfv_h \in \; \bfH(\div ;K)$ to $\beta^{-1}_h\bfv_h \in\bfH(\curl;K)$. 

Then, the $\bfH(\div)$ IFE space is defined as
\begin{equation}
\begin{split}
\label{IFE_loc_spa_div}
\bfS^f_h(K) = \{  \bfv_h|_{K^{\pm}_h} \in \mathcal{RT}_0(K^{\pm}_h)~: ~ \bfv_h ~ \text{satisfies} ~ \eqref{weak_jcdiv} \}.
\end{split}
\end{equation}
Again, we can derive the explicit formulas for functions in $\bfS^f_h(K)$:
\begin{equation}
\label{IFE_div_represent}
\bfv_h = 
\begin{cases}
      & M' \bfc +  c_0 (\bfx - \bfx^m)\,\,\,\,\, \text{in} \, K^+_h, \\
      & \,\,\,\,\,\, \bfc + c_0 (\bfx - \bfx^m)  \,\,\,\,\, \text{in} \, K^-_h,
\end{cases}
\end{equation}
where $c_0$ and $\bfc$ are arbitrary scalar- and vector-valued constants, and $M' = \rho^{-1} M$.
\RG{\begin{remark}
Comparing the virtual element space \eqref{virtual_space_hcurl_loc} and the IFE space \eqref{IFE_loc_spa}, we find that the only difference is that $\beta_h \bfv_h\notin \bfH(\div;K)$ for $\bfv_h\in \bfS^e_h(K)$, since the normal continuity only holds at one point $\bf x^m$ as shown in \eqref{weak_jc_3}. Thus, $\bfS^e_h(K)\not\subset \bfV^e_h(K)$ which is different from the $H^1$ case. Similarly, for the $\bfH(\div)$ case, we still do not have $\bfS^f_h(K) \not\subset \bfV^f_h(K)$ since the tangential continuity only holds at $\bfx^m$. 
Note that the similar practices occur in the VEM literature. 
For example, the serendipity VEM spaces often use DoFs/projections as constraints in the definition of the virtual element spaces to eliminate interior DoFs, e.g., \cite{BeiraodaVeigaBrezziEtAl2018Family,2020BeiroMascotto}. However, due to the presence geometry-tied constraints such as the barycenter in the space definition, some common constructions for the vector polynomial space may not directly yield a subspace of this serendipity-type space anymore. Nevertheless, the flexibility of the VEM framework still guarantees convergence for a class of admissible geometry-tied constraints if the polynomial space offers approximation, e.g., see the discussion in \cite[Appendix]{2022CaoChenGuo}. Another example is VEM on curved edges or faces, e.g., \cite{BeiraodaVeigaRussoEtAl2019virtual}, the exact geometry is captured by the virtual element spaces that does not contain the standard polynomial spaces, and the projection is done in an isogeometric fashion to guarantee the approximation to geometry.
As for the present case, the IVE spaces contain a piecewise constant vector proper subspace of the IFE spaces, onto which the IVE functions are then projected. This is sufficient for an optimal first order accuracy.
\end{remark}
}

\subsubsection{The Exact Sequence for IFE Spaces}
First of all, it is not hard to see
\begin{equation}
\label{curlVS}
\curl \, \bfV^e_h = \curl\, \bfS^e_h(K) = Q^{\alpha_h}_h(K) ~~~ \text{and} ~~~ \div \, \bfV^f_h = \div\, \bfS^e_h(K) = Q^{1}_h(K).
\end{equation}

Let us recall the discrete de Rham complex and exact sequence for IFE spaces which will be useful in the later discussion. Here, we only need the local ones:~\cite[Theorem 3.5]{2020GuoLinZou} shows
\begin{equation}
\label{thm_DR_3_curl}
\left.\begin{array}{ccccc}
\mathbb R \xrightarrow[]{\hookrightarrow} S^{n}_h(K)  \xrightarrow[]{~~\nabla~~}  \bfS^e_h(K)  \xrightarrow[]{~~\curl~~} Q^{\alpha_h}_h(K) \xrightarrow[]{} 0 .
\end{array}\right.
\end{equation}
A similar exact sequence is
\begin{equation}
\label{thm_DR_3_div}
\left.\begin{array}{ccccc}
\mathbb R \xrightarrow[]{\hookrightarrow}  \widetilde{S}^{n}_h(K)  \xrightarrow[]{~~\bcurl~~}  \bfS^f_h(K)  \xrightarrow[]{~~~\div~~} Q^{1}_h(K) \xrightarrow[]{} 0 .
\end{array}\right.
\end{equation}
Here, $\widetilde{S}^n_h(K)$ is an $H^1$ IFE spaces but with the parameter $\beta_h^{-1}$ and a rotated gradient, i.e., \eqref{eq:jump-1} is replaced by 
\begin{equation}\label{eq:curljump-2}  
  ( \beta^{+}_h)^{-1}\bcurl v^{+}_h \cdot {\bar{\mathbf{t}}}  = (\beta^{-}_h)^{-1}\bcurl v^{-}_h \cdot {\bar{\mathbf{t}}}   \quad \text{ on } \; \Gamma^K_h. 
\end{equation}
We mention that $\bfS^f_h(K)$ is the space used in~\cite{2021Ji} for mixed IFE methods. Then, we have the following result. 

\begin{lemma}
\label{lem_ife_h1_hucl_hdiv}
The Hodge star operator $\beta_h\cdot$ induces a one-to-one mapping from $\nabla S^n_h(K)$ to $\bcurl\, \widetilde{S}^n_h(K)$:
\begin{equation}
\label{ife_h1_hucl_hdiv}
\beta_h \nabla S^n_h(K) = \bcurl\, \widetilde{S}^n_h(K).
\end{equation}
\end{lemma}
\begin{proof}
For a function $v_h\in S^n_h(K)$, $\nabla v_h$ is a piecewise constant vector in $\bfH(\curl; K)$, i.e., with tangential continuity. By construction $\beta_h \nabla v_h\in \bfS^f_h(K)$ is a piecewise constant vector but now continuous at normal direction. Therefore, $\div \beta_h \nabla v = 0$. So we have proved $\beta_h \nabla S^n_h(K) \subseteq \text{Ker}(\div)\cap \bfS^f_h(K) = \bcurl\, \widetilde{S}^n_h(K)$. By the same argument but switching $S^n_h(K) $ and $\widetilde{S}^n_h(K)$, we have $\beta_h^{-1} \bcurl\, \widetilde{S}^n_h(K) \subseteq \text{Ker}(\curl)\cap \bfS^e_h(K)  = \nabla S^n_h(K)$. This finishes the proof.
\end{proof}

\begin{remark}
\label{lem_varphi_stab}
Lemma \ref{lem_hcurl_wellposed} shows for each $\bfv_h\in \bfV^e_h$, there uniquely exists $\varphi_h\in H^1(K)$ such that $\beta^{-1}_h\bcurl\,\varphi_h = \bfv_h$ and $\int_{\partial K} \varphi_h \dd s = 0$. If $\bfv_h$ is assumed to be a constant vector whose divergence vanishes, then by sequence \eqref{thm_DR_3_div} $\varphi_h\in \widetilde{S}^n_h(K)$. Moreover, with the \RG{Poincar\'e-Friedrichs'} inequality \eqref{Poincare_eq2} and the trace inequality in Lemma \ref{H1 trace}, we can show the stability:
\begin{equation}
\label{varphi_stab}
h^{1/2}_K\| \varphi_h \|_{0,\partial K} + \| \varphi_h \|_{0,K} \lesssim h_K \| \bfv_h \|_{0,K}.
\end{equation}
\end{remark}

\subsection{Projections}
It can be shown that the IFE spaces $S_h^n(K)$, $\bfS^e_h(K)$ and $\bfS^f_h(K)$ are unisolvent by the nodal DoFs~\cite{Guo;Lin:2019immersed}, edge DoFs $\int_e \bfv_h\cdot\bft \dd s$~\cite{2020GuoLinZou} and $\int_e \bfv_h\cdot\bfn \dd s$~\cite{2021Ji}, respectively. These DoFs are critical for the conventional IFE methods in both analysis and computation. Proofs of the unisolvence with respect to the DoFs are generally very technical and rely on mesh assumption, for example the ``no-obtuse-angle'' condition introduced in~\cite{2020GuoLinZou,2021Ji}. For some other problems, the unisolvence may not even hold, such as the elasticity problem~\cite{2017GuoLinElas}, or the case that the interface intersects an element multiple times. It is highlighted that both the analysis and implementation of the proposed method do not rely on the unisolvence of the DoFs for the IFE spaces themselves, 
as they only serve as a computable projection space of the underlying virtual element spaces that offers a sufficient approximation power. \LC{IFE is used locally and thus no inter-element continuity is needed.} 
Roughly speaking, the usual IFE shape functions will be replaced by a certain projection of $\phi_h$ to IFE spaces, where $\phi_{h}$'s are the shape functions of the virtual element spaces. This is one of the major difference of the proposed method from those classical IFE works. With this property, the IVEM is more flexible and generalizable.

Let us describe how to compute the projection from the IVE spaces to the IFE spaces. For the $H^1$ case, we introduce a projection $\Pi^{\beta_h}_K :\LC{H^1(K)}\to S_h^{n}(K)$:
\begin{equation}
\label{eq:ife-projection_1}
\begin{aligned}
(\beta_h \nabla \Pi^{\beta_h}_K u, \nabla v_h)_K = (\beta_h \nabla u, \nabla v_h)_K, \quad \forall v_h\in S^n_h(K),
\quad \text{ and }\;   \int_{\partial K} \big( u-\Pi^{\beta_h}_K u_h \big)\, \dd s=0.
\end{aligned}
\end{equation}
By the continuity of $u_h\in V^n_h(K)$ and flux jump condition of $v_h \in S^n_h(K)$, applying integration by parts, we have
\begin{equation}
\label{eq:ife-projection_2}
\int_K \beta_h \nabla u_h\cdot \nabla v_h \dd \bfx = \int_{\partial K} \beta_h u_h \nabla v_h\cdot \bfn \dd s,
\end{equation}
which is computable, since $u_h|_{\partial K}$ is explicitly known, and $v_h \in S^n_h(K)$ can have its gradient evaluated explicitly. \LC{Therefore $\Pi^{\beta_h}_K u_h$ for a VEM function $u_h\in V^n_h(K)$ is computable}. This projection exactly mimics the usual one used in the VEM literature. 

For the $\bfH(\curl)$ interface problem, \LC{as $\curl \bfv_h$ is explicitly computable through the DoFs, cf. \eqref{curl_val}}, but not $\bfv_h$. As a result, we only need to approximate the $L^2$ term. To this end, a weighted $L^2$ projection is introduced $\bfPi^{\beta_h}_K:L^2(K) \rightarrow \nabla S^n_h(K)$. \LC{For $\bfu \in L^2(K)$, $\beta_h \bfPi^{\beta_h}_K \bfu\in \nabla S^n_h(K)$ such that}
\begin{equation}
\label{ife_projection_curl_1}
(\beta_h \bfPi^{\beta_h}_K \bfu, \bfv_h)_K  = (\beta_h  \bfu, \bfv_h)_K, ~~~ \forall \bfv_h\in \nabla S^n_h(K).
\end{equation}
Since $\bfv_h \in \nabla S^n_h(K)$, by \eqref{ife_h1_hucl_hdiv} we have $\beta_h\bfv_h\in \beta_h \nabla S^n_h(K) = \bcurl\, \widetilde{S}^n_h(K)$. Hence, there exists $\varphi_h\in \widetilde{S}^n_h(K)$ such that $\bcurl\, \varphi_h = \beta_h\bfv_h$. In particular, we can use \eqref{IFE_space_h1_represnt} to express $\varphi_h$ as 
\begin{equation}
\label{IFE_space_h1_represnt_new}
\varphi_h(\bfx) =   (R_{-\frac{\pi}{2}}\beta_h\bfv_h) \cdot (\bfx-\bfx^m) + c_0 ,
\end{equation} 
where $R_{-\frac{\pi}{2}}$ is the counterclockwise $\frac{\pi}{2}$ rotation matrix, and $c_0$ can be taken as an arbitrary constant with respect to which the projected vector is invariant. 
Then, \LC{for $\bfu_h\in \bfV^e_h$}, it follows from integration by parts that
\begin{equation}
\label{ife_projection_curl_2}
\int_K \beta_h \bfPi^{\beta_h}_K \bfu_h \cdot \bfv_h \dd \bfx = \int_K  \bfu_h\cdot \bcurl\, \varphi_h \dd \bfx =  \int_K  \curl\,\bfu_h \,  \varphi_h \dd \bfx - \int_{\partial K}  \bfu_h\cdot\bft\, \varphi_h \dd s,
\end{equation}
where $\curl \bfu_h$ is computable through DoFs as shown in \eqref{curl_val}. Notice that as $[\bfu_h\cdot \bar{\bft}] = 0$ and $\varphi_h$ is continuous on $\Gamma_h^K$, there is no contribution from the integral on $\Gamma_h^K$. 

\RG{
The projection for the $\bfH(\div)$ case is defined similarly. A weighted $L^2$ projection is introduced $\widetilde{\bfPi}^{\beta^{-1}_h}_K:\bfV^f_h(K) \rightarrow \curl \widetilde{S}^n_h(K)$:
\begin{equation}
\label{ife_projection_div_1}
(\beta^{-1}_h \widetilde{\bfPi}^{\beta^{-1}_h}_K \bfu_h, \bfv_h)_K  = (\beta^{-1}_h  \bfu_h, \bfv_h)_K, ~~~ \forall \bfv_h\in \curl \widetilde{S}^n_h(K).
\end{equation}
Given $\bfv_h \in \curl \widetilde{S}^n_h(K)$, there exists $\varphi_h\in {S}^n_h(K)$ such that $\nabla \, \varphi_h = \beta^{-1}_h\bfv_h$. 
\begin{equation}
\label{ife_projection_div_2}
\int_K \beta^{-1}_h \bfPi^{\beta^{-1}_h}_K \bfu_h \cdot \bfv_h \dd \bfx = \int_K  \bfu_h\cdot \nabla \, \varphi_h \dd \bfx = - \int_K  \div\,\bfu_h \,  \varphi_h \dd \bfx + \int_{\partial K}  \bfu_h\cdot\bfn\, \varphi_h \dd s,
\end{equation}
where $\div\bfu_h $ can be computed through \eqref{vem_hdiv_comput} with DoFs and $\bfu_h\cdot\bfn$ are the given DoFs.
}

In the rest of this article, for the sake of simplicity, we shall drop $\beta_h$ of the projections $\Pi^{\beta_h}_K$ and $\bfPi^{\beta_h}_K$, and furthermore $\Pi_K$ and $\bfPi_K$, regardless of being interface element or not, are adopted to maintain a consistent and concise set of notation. On each non-interface element, the projection is simply the identity operator.

\section{Properties of IFE Functions}

In this section, we recall some properties for IFE functions and show some novel ones to be used. In the following discussion, any subdomain $D\subseteq\Omega$, we denote for simplicity
\begin{align*}
\| u \|_{E,k,D} := \| u^{+}_E \|_{k,D} + \| u^{-}_E \|_{k,D}  \quad \text{and} \quad
\| \bfu \|_{E,\curl,k,D} := \| \bfu \|_{E,k,D} + \| \curl\bfu \|_{E,k,D},
\end{align*}
where $k$ is a non-negative constant, and $u_E^{\pm}$ are the Sobolev extensions defined before Theorem \ref{thm_ext}. For scalar- or vector-valued functions, their corresponding seminorms adopt this notation convention as well. We also need the patch of an interface element $K$ which is the collection of elements neighboring $K$:
$$
\omega_K:=\bigcup_{T\in \mathcal T_h, \overline{K}\cap\overline{T}\neq \emptyset} T,\quad 
\text{ and } \quad \omega_K^{\pm} := \omega_K\cap \Omega^{\pm}.
$$
\RG{In the following discussion, we focus our analysis on interface element where the specially constructed IVE and IFE spaces are used. The analysis on non-interface elements are trivial since the standard FE functions are used.}

\subsection{The $H^1$ IFE Functions}
We first recall the trace inequalities for the $H^1$ IFE functions.
\begin{lemma}[A trace inequality for $H^1$ IFE functions~\cite{2015LinLinZhang}]
\label{IFE trace}
For each interface element $K$ and its edge $e$, there holds
\begin{equation}
\label{IFE trace eq0}
h^{1/2}_K\| \nabla v_h\|_{0,e}\lesssim \| \nabla v_h\|_{0,K}, \quad\quad \forall v_h\in S^n_h(K),
\end{equation}
\LC{where the constant hidden in $\lesssim$ is independent of the location of the interface.}
\end{lemma}
\RG{Result \eqref{IFE trace eq0} is non-trivial in the sense that the hidden constant may depend on the interface location if classic tools are applied on each subelement. In particular, the constant may blow up when the cut subelement is degenerated. We refer readers to \cite[Section 3.1]{2015LinLinZhang} for a detailed proof. Heuristically for IFE functions, $\nabla v_h$ is a piecewise constant, and \eqref{IFE trace eq0} is possible through scaling arguments. For IVE function $v_h$, however, such trace result may not be easy to establish as $\nabla v_h$ is non-polynomial in general and extra geometric conditions are needed, cf. \cite{chen2018some}. This is also the case for the $\bfH(\curl)$ IFE functions given in Lemma \ref{lem_trace_inequa}.}

Let us then discuss the approximation results for the projection $\Pi_K$ defined by 
\eqref{eq:ife-projection_1}. 
Similar to the standard $H^1$ projection, with the known approximation results for IFE interpolations in the literature~\cite{Guo;Lin:2019immersed,guzman2017finite}, the results for $\Pi_K$ may directly follow from the best approximation property of the projection. However, we shall see that the analysis further demands the approximation of each polynomial component of $\Pi_Ku$ on the whole element $K$. \LC{Recall that $\Pi_K u$ is piecewise linear in $K^{\pm}$ satisfying the jump condition \eqref{eq:jump-1}.} With a slight abuse of notation, we consider the two polynomial extensions of $\Pi_K u|_{K^{\pm}}$ defined on the entire element $K$
\begin{equation}
\label{proj_pm}
\Pi^{\pm}_K u := (\Pi_K u)^{\pm}_E,
\end{equation}
where $ (\Pi_K u)^{\pm}_E$ are trivial extensions of $\Pi_K u|_{K^{\pm}}$. Namely, we need to estimate $u^{\pm}_E - \Pi^{\pm}_Ku$ on the entire element. \LC{See Fig. \ref{fig:interface-diff-2} for an illustration.} 

For this purpose, we need to employ a quasi-interpolation operator introduced in~\cite{guzman2017finite} as an intermediate tool which is denoted by $J_Ku$. But, since our IFE functions are defined with approximate interface $\Gamma_h$, we need to slightly modify the definition here. Define the interpolation operator $J_K$ such that
\begin{equation}
\label{quaInterp_1}
J_K u =
\begin{cases}
      & J_K^+u, ~~~ \text{in} ~ \omega^+_K, \\
      & J_K^-u, ~~~ \text{in} ~ \omega^-_K,
\end{cases}
\end{equation}
where $J_K^{\pm}u$ are two linear polynomials satisfying the following conditions 
\begin{subequations}
\label{quaInterp_2}
\begin{align}
    & J^-_Ku|_{\Gamma^K_h} =  J^+_Ku|_{\Gamma^K_h} := \pi_{\omega_K} u^+_E|_{\Gamma^K_h}, \label{quaInterp_2_1} \\
    & \beta^-_h \nabla J^-_Ku\cdot\bar{\bfn}_K =  \beta^+_h \nabla J^+_Ku\cdot\bar{\bfn}_K := \beta^- \nabla \pi_{\omega_K} u^-_E, \label{quaInterp_2_2}
\end{align}
\end{subequations}
where $\pi_{\omega_K}$ is the standard $L^2$ projection onto $\mathbb{P}_1(\omega_K)$. We note that the only difference between $J_K$ and the one in \cite{guzman2017finite} (denoted by $I_T$ in (3.4) therein) is that the jump conditions are imposed on $\Gamma^K_h$.

Similar to \eqref{proj_pm}, we denote the two polynomials that are trivial $H^2$-extensions of $J_K u|_{K^{\pm}}$ still as $J^{\pm}_Ku$, which are defined on the whole element $K$.
\LC{Roughly speaking, \eqref{quaInterp_2} defines a piecewise linear polynomial $J_K u$ by a Hermite interpolation at a point on $\Gamma_h$. Moreover, by an averaging type Taylor expansion,} these two polynomials have the desired optimal approximations to their corresponding functions $u^{\pm}_E$ on the whole element. This crucial property is given by the lemma below, and serves as the key in our analysis. 

\begin{lemma}%
\label{quasi IFE estimate}
For $u\in H^2(\beta;\mathcal T_h)$, on any $K\in \mathcal{T}^i_h$ there holds
\begin{equation}
    |u^{\pm}_E - J^{\pm}_K u |_{1,K}
    \lesssim h_K  \| u \|_{E,2,\omega_K} .
\end{equation}
\end{lemma}
\begin{proof}
 The argument is the same as Lemmas 3-5 in \cite{guzman2017finite}.
\end{proof}

A similar estimate for $\Pi^{\pm}_K$ can be established on the whole element $K$. The analysis needs to employ the quasi interpolation $J_K^{\pm} u$ as an intermediate quantity to bridge the estimate.

\begin{lemma}
\label{lem_ife_projection}
For $u\in H^2(\beta;\mathcal T_h)$, on any $K\in \mathcal{T}^i_h$ there holds
\begin{equation}
\label{lem_ife_projection_eq0}
| u^{\pm}_E -\Pi^{\pm}_K u |_{1,K}\lesssim h_K \|u \|_{E,2,\omega_K} +   |u|_{E,1,\delta{K}} .
\end{equation}
\end{lemma}
\begin{proof}
By the triangle inequality and Lemma \ref{quasi IFE estimate}, it suffices to estimate the difference $| J_K^{\pm}u^{\pm} -\Pi^{\pm}_K u |_{1,K}$. Without loss of generality, we only discuss the $+$ piece. We have the following trivial split
\begin{equation}
\label{lem_ife_projection_eq2}
| J_K^+u -\Pi^{+}_K u |_{1,K} \lesssim \underbrace{ | J_K^+u -\Pi^{+}_K u |_{1,K^+_h} }_{({\rm I})}  + \underbrace{ | J_K^+u -\Pi^{+}_K u |_{1,K^-_h}}_{({\rm II})}.
\end{equation}
The estimate for $({\rm I})$ \LC{ is relatively easy as the domain $K_h^+$ matches the definition of $\Pi_K^+$.} By the triangle inequality, 
\begin{equation}
\begin{split}
\label{lem_ife_projection_eq3}
 | J_K^+u -\Pi^{+}_K u |_{1,K^+_h} & \lesssim  | u -\Pi^{+}_K u |_{1,K^+_h} +  | u - J_K^+u |_{1,K^+_h}  \\
 & \lesssim | u -\Pi_K u |_{1,K} +  | u - J_Ku |_{1,K} \\
 & \lesssim  | u - J_Ku |_{1,K} \lesssim      |u^{\pm}_E - J^{\pm}_K u |_{1,K} +  | u |_{E,1,\delta K}
 \end{split}
\end{equation}
where in the third inequality we have used the best approximation property for $\Pi_K$ under the energy norm which is equivalent to the $|\cdot|_{1,K}$ norm. 

\LC{The second term $({\rm II})$ is to estimate the error when the domain $K_h^-$ is out of the part defining $\Pi_K^+$. Again we refer to Fig. \ref{fig:interface-diff-2} for an illustration.}
By the jump conditions on $\Gamma^K_h$ and employing the matrix in \eqref{gradient_ife_2}, we have the following identity for gradients of an IFE function $v_h\in S^{n}_h(K)$: $\nabla v_h^+ =  M \nabla v_h^-$ with $M$ given in \eqref{gradient_ife_2}.
It clearly shows $\| \nabla v_h^+ \| \simeq \| \nabla v_h^- \|$,
where $\|\cdot\|$ are just Euclidean norms for vectors, and the hidden constant depends on $\beta$ through the eigenvalues of $M$. Therefore, by letting $v_h = J_K u - \Pi_K u$, we have
 \begin{equation}
\begin{split}
\label{lem_ife_projection_eq5}
 | J_K^{+} u-\Pi_K^{+} u |_{1,K^{-}_h}
 & \lesssim | J_K^{-} u-\Pi_K^{-} u |_{1,K^{-}_h} 
\end{split}
\end{equation}
where the later one can be proved similarly to \eqref{lem_ife_projection_eq3}.
\end{proof}

\begin{figure}[htbp]
\begin{center}
\begin{subfigure}[b]{0.225\linewidth}
      \centering
      \includegraphics[width=0.95\textwidth]{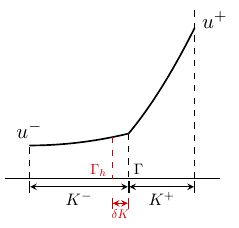}
      \caption{}
      \label{fig:interface-diff-0}
\end{subfigure}
\begin{subfigure}[b]{0.225\linewidth}
      \centering
      \includegraphics[width=0.95\textwidth]{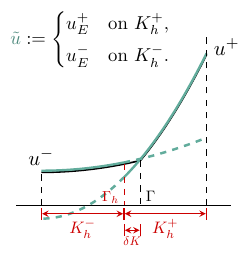}
      \caption{}
      \label{fig:interface-diff-1}
\end{subfigure}
\begin{subfigure}[b]{0.225\linewidth}
  \centering
  \includegraphics[width=0.95\textwidth]{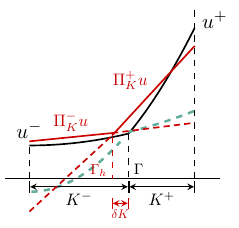}
  \caption{}
  \label{fig:interface-diff-2}
\end{subfigure}%
\begin{subfigure}[b]{0.225\linewidth}
  \centering
  \includegraphics[width=\textwidth]{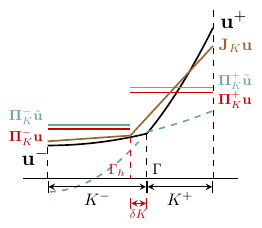}
  \caption{}
  \label{fig:interface-diff-3}
\end{subfigure}%
\end{center}
\caption{\LC {A 1D analog of the comparison used in Lemma \ref{lem_ife_projection} and Lemma \ref{lem_utilde}:
\eqref{fig:interface-diff-0}--\eqref{fig:interface-diff-1}: $u$ and $\tilde{u} := u_E^{\pm}$ on $K^{\pm}_h$; \eqref{fig:interface-diff-2}: $\Pi_K^{\pm}u$ for $H^1$ function in Lemma \ref{lem_ife_projection}; \eqref{fig:interface-diff-3}: $\Pi_K^{\pm} \bfu$ versus $\Pi_K^{\pm} \tilde{\bfu}$ for $\bfH(\curl)$ case, where scalar functions in  this figure is illustrated as the lateral view of the tangential component of the vector functions.}
}
\label{fig:interface-diff}
\end{figure}

\subsection{The $\bfH(\curl)$ IFE Functions}
The similar situation also exists for the $\bfH(\curl)$ case, i.e., we need the estimates for the two polynomial components of \LC{the weighted $L^2$ projection} $\bfPi^{\pm}_K\bfu$ on the entire element (the notation is similar to \eqref{proj_pm}). \RG{In this case, we employ the quasi interpolation defined in \cite[(4.4)]{2020GuoLinZou} as an intermediate estimate in the error analysis, which is similar to that for \eqref{quaInterp_1}.} Here we denote it as $\bfJ_K$ to be distinguished from the $H^1$ scalar case of which the approximation is recalled below:
\begin{lemma}[Theorem 4.1 in~\cite{2020GuoLinZou}]
\label{quasi_curl}
For $\bfu\in \bfH^1(\emph{curl},\alpha,\beta;\mathcal{T}_h)$, on any $K\in \mathcal{T}^i_h$ there holds
\begin{equation}
\label{quasi_curl_eq0}
\| \bfu^{\pm}_E - \bfJ^{\pm}_K \bfu  \|_{\bfH(\emph{curl};K)} \lesssim h_K \| \bfu \|_{E,\curl,1,\omega_K} .
\end{equation}
\end{lemma}

We also recall the following result.
\begin{lemma}[Lemma 4.2 in ~\cite{2020GuoLinZou} and Lemma 5.4 in \cite{2021CaoChenGuo}]
\label{lem_u_gam_h}
For $\bfu\in \bfH^1(\emph{curl},\alpha,\beta;\mathcal{T}_h)$, on any $K\in \mathcal{T}^i_h$,
the difference of the extensions on the approximate interface $\Gamma^K_h$ along the tangential direction $\bar{\bft}$ satisfies
\begin{equation}
\label{lem_u_gam_eq0}
\| \bfu^+_E\cdot\bar{\bft} - \bfu^-_E\cdot\bar{\bft} \|_{0, K} \lesssim h_K \| \bfu\|_{E,1,\omega_K}.
\end{equation}
\end{lemma}
\LC{As a function in $\bfH(\curl)$, $\bfu\cdot \bft$ is continuous. Extension will preserve the tangential continuity.}
\RG{Note that $\bfu^+_E\cdot\bft=\bfu^-_E\cdot\bft$ on $\Gamma\cap K$, then it is reasonable to expect $\bfu^+_E\cdot\bar{\bft}$ is close to $\bfu^-_E\cdot\bar{\bft}$ on $\Gamma^K_h$ as $\bar{\bft}$ is a good approximation to $\bft$. As both the two quantities $\bfu^{\pm}_E\cdot\bar{\bft}$ are well-defined on the entire element $K$, the estimate in \eqref{lem_u_gam_eq0} is a Poincar\'e-type inequality in a certain sense. }

Then we can show the estimates for $\bfPi_K\bfu$ and $\curl\,\bfu_I$. For the $\curl$ case, we need to eliminate the mismatch term on $\delta K$ in the error bound. For this purpose, we note that the mismatched term is essentially caused by the fact that $\bfu$ itself is partitioned by $\Gamma$ but $\bfPi_K\bfu$ and $\bfu_I$ are partitioned by $\Gamma_h$. So, it inspires us to introduce a new function $\tilde{\bfu}:=\bfu^{\pm}_E$ on $K^{\pm}_h$ as an intermediate quantity, and present the following estimate. \RG{Note that $\tilde{\bfu} = \bfu$ on $\partial K$, and thus $\tilde{\bfu}_I = \bfu_I$. In fact, $\tilde{\bfu}$ differs from $\bfu$ only on the mismatched region $\delta K$, i.e., on $\delta K\cap K_h^+$, $\bfu - \tilde{\bfu} = \bfu -  \bfu^+_E = \bfu^-_E - \bfu^+_E$, and similarly on $\delta K\cap K_h^-$, $\bfu - \tilde{\bfu} = \bfu^+_E - \bfu^-_E$.} We also note that $[\tilde{\bfu}\cdot \bar \bft]\mid_{\Gamma_h^K}\neq 0$. For an analog of this heuristic in a 1-dimensional setting, please refer to Figure \ref{fig:interface-diff-3}.

\begin{lemma}
\label{lem_utilde}
Let $\tilde{\bfu}:=\bfu^{\pm}_E$ on $K^{\pm}_h$, then there holds
\begin{subequations}
\label{utilde_eq0}
\begin{align}
\| \bfPi^{\pm}_K\bfu - \bfPi^{\pm}_K\tilde{\bfu} \|_{0,K} & \lesssim  h_K \| \bfu \|_{E,\curl,1,\omega_K} ,  \label{utilde_eq01} \\
\curl\, \bfu_I &= \curl\, \tilde{\bfu}_I .\label{utilde_eq02}
\end{align}
\end{subequations}
\end{lemma}
\begin{proof}
Since $\bfu$ and $\tilde{\bfu}$ match on $\partial K$, \eqref{utilde_eq02} is trivial from integration by parts. We estimate $\bfPi_K\bfu - \bfPi_K\tilde{\bfu}:=\bfw_h$. Given each $\bfv_h\in \nabla S^n_h(K)$, by \eqref{ife_h1_hucl_hdiv}, we can find $\varphi_h\in\widetilde{S}^n_h(K)$ satisfying $\bcurl\,\varphi_h= \beta_h\bfv_h$ given in Remark \ref{lem_varphi_stab}. %
 For $\bfv= \bfu$ or $\tilde{\bfu}$, a similar formula to \eqref{ife_projection_curl_2} leads to
\begin{equation}
\label{Pi_curl_est_eq-1}
\int_K \beta_h \bfPi_K \bfv \cdot \bfv_h \dd \bfx =  \int_K  \curl\,\bfv \,  \varphi_h \dd \bfx - \int_{\partial K} \bfv \cdot\bar{\bft} \varphi_h \dd s - \int_{\Gamma^K_h}[\bfv\cdot\bar{\bft}]\varphi_h \dd s
\end{equation}
where the last term vanishes for $\bfv= \bfu$. 
Using the fact that $\bfu$ and $\tilde{\bfu}$ match on $\partial K$ and taking the difference of \eqref{Pi_curl_est_eq-1} for $\bfv= \bfu$ and $\tilde{\bfu}$, we have
\begin{equation}
\label{Pi_curl_est_eq-2}
\int_K \beta_h \bfw_h\cdot\bfv_h \dd \bfx =  \underbrace{ \int_{\delta{K}}  \curl\,(\bfu - \tilde{\bfu}) \,  \varphi_h \dd \bfx }_{({\rm I})} -  \underbrace{ \int_{\Gamma^K_h}(\bfu^+_E\cdot\bar{\bft} - \bfu^-_E\cdot\bar{\bft})\varphi_h \dd s }_{({\rm II})}.
\end{equation}
By H\"older's inequality and \eqref{varphi_stab}, we have
\begin{equation}
\label{Pi_curl_est_eq-3}
({\rm I}) \lesssim \| \curl\,(\bfu^+_E - \bfu^-_E) \|_{0, \delta{K}} \| \varphi_h \|_{0, K} \lesssim h_K \| \curl\,(\bfu^+_E - \bfu^-_E) \|_{0, \delta{K}} \| \bfv_h \|_{0, K}.
\end{equation}
Using the trace inequality in Lemma \ref{trace_surf}, \eqref{varphi_stab}, and $\bcurl\,\varphi_h= \beta_h\bfv_h$ yields
\begin{equation}
\label{Pi_curl_est_eq-4}
\begin{split}
({\rm II}) &\lesssim ( h^{-1/2}_K\| \bfu^+_E\cdot\bar{\bft} - \bfu^-_E\cdot\bar{\bft} \|_{0, K} + h^{1/2}_K | \bfu^+_E\cdot\bar{\bft} - \bfu^-_E\cdot\bar{\bft} |_{1, K} ) \cdot( h^{-1/2}_K \| \varphi_h \|_{0, K} +  h^{1/2}_K|\varphi_h|_{1,K})  \\
&\lesssim  \| \bfu^+_E\cdot\bar{\bft} - \bfu^-_E\cdot\bar{\bft} \|_{0, K} \| \bfv_h \|_{0, K} + h_K | \bfu^+_E\cdot\bar{\bft} - \bfu^-_E\cdot\bar{\bft} |_{1, K} \| \bfv_h \|_{0, K} \\
&\lesssim h_K\| \bfu \|_{E,1,\omega_K} \| \bfv_h \|_{0, K}
\end{split}
\end{equation}
where we have used Lemma \ref{lem_u_gam_h} in the third inequality. Putting \eqref{Pi_curl_est_eq-3} and \eqref{Pi_curl_est_eq-4} into \eqref{Pi_curl_est_eq-2}, letting $\bfv_h = \bfw_h$, and cancelling one $\| \bfw_h \|_{0,K}$ on each side, we obtain
\begin{equation}
\begin{split}
\label{Pi_curl_est_eq-5}
\| \bfPi_K\bfu - \bfPi_K\tilde{\bfu} \|_{0,K} \lesssim  h_K( \| \curl\,(\bfu^+_E - \bfu^-_E) \|_{0, \delta{K}} + \| \bfu \|_{E,1,\omega_K} ).
\end{split}
\end{equation}
Note that $\bfPi_K\bfu - \bfPi_K\tilde{\bfu} \in \nabla S^n_h(K)$ by the exact sequence. So, using the argument similar to \eqref{lem_ife_projection_eq5}, we directly induce \eqref{utilde_eq01} from \eqref{Pi_curl_est_eq-5}.
\end{proof}

With this preparation, we will present the following crucial estimate.
\begin{lemma}
\label{Pi_curl_est}
For $\bfu\in \bfH^1(\emph{curl},\alpha,\beta;\mathcal{T}_h)$, on any $K\in \mathcal{T}^i_h$ there holds
\begin{subequations}
\label{Pi_curl_est_eq0}
\begin{align}
 \| \bfPi^{\pm}_K \bfu - \bfu^{\pm}_E \|_{0,K}     & \lesssim h_K \| \bfu \|_{E,\curl,1,\omega_K} ,  \label{Pi_curl_est_eq01} \\
  \| \curl\,\bfu^{\pm}_E - \curl^{\pm} \bfu_I \|_{0,K}    & \lesssim h_K \| \bfu \|_{E,\curl,1,\omega_K}  , \label{Pi_curl_est_eq02}
\end{align}
\end{subequations}
where $\curl^{\pm} \bfu_I = (\curl \, \bfu_I)^{\pm}$ are the two constants used on the whole element.
\end{lemma}
\begin{proof}
The argument is similar to Lemma \ref{lem_ife_projection} but slightly more complicated, since we need to avoid the mismatched region $\delta{K}$ by employing the function $\tilde{\bfu}$ introduced in Lemma \ref{lem_utilde}. We decompose the argument into several steps. 

Step 1. We show
\begin{equation}
\label{Pi_curl_est_eq1}
\| \tilde{\bfu} - \bfPi_K \tilde{\bfu} \|_{0,K} \lesssim h_K \| \bfu \|_{E,\curl,1,\omega_K} .
\end{equation}
Thanks to \eqref{IFE_curl_represent}, we can write $\bfJ_K\bfu$ as
\begin{equation}
\label{Pi_curl_est_eq2}
\bfJ_K \bfu = \bfp + p_0 (-(x_2 - x^m_2), x_1-x^m_1)^{\intercal},
\end{equation}
where $p_0$ and $\bfp$ are piecewise scalar- and vector-valued constants. In particular, we have $\bfp\in\nabla S_h(K)=\text{Ker}(\curl)\cap \bfS^e_h(K)$, and $p_0= \curl\, \bfJ_K\bfu/2$. Then, by the best approximation property of the projection, we have
\begin{equation}
\begin{split}
\label{Pi_curl_est_eq3}
\| \tilde{\bfu} - \bfPi_K \tilde{\bfu} \|_{0,K} & \lesssim \| \sqrt{\beta_h}(\tilde{\bfu} - \bfPi_K \tilde{\bfu} )\|_{0,K}\lesssim \| \sqrt{\beta_h}(\tilde{\bfu} - \bfp )\|_{0,K}  \lesssim \| \tilde{\bfu} - \bfJ_K\bfu \|_{0,K} + h_K \| \curl\,\bfJ_K \bfu \|_{0,K},
\end{split}
\end{equation}
where in the last inequality we have inserted $ p_0 (x_2 - x^m_2, -(x_1-x^m_1))^{\intercal}$. Noticing that the partition of $\tilde{\bfu}$ exactly matches $\bfJ_K \bfu$, i.e., both of their piecewise definitions are separated by $\Gamma^K_h$. Hence, applying Lemma \ref{quasi_curl} yields \eqref{Pi_curl_est_eq1}.

\vspace{0.1in}

Step 2. We refine the estimate in \eqref{Pi_curl_est_eq1} to the entire element; namely, with $\tilde{\bfu}^{\pm}=\bfu^{\pm}_E$, we need to show
\begin{equation}
\label{Pi_curl_est_eq3-1}
\| \bfu^{\pm}_E - \bfPi^{\pm}_K \tilde{\bfu} \|_{0,K} \lesssim h_K \| \bfu \|_{E,\curl,1,\omega_K} .
\end{equation}
Without loss of generality, we focus on $\bfPi^{+}_K \tilde{\bfu} - \bfu^{+}_E$. Similar to the argument in Lemma \ref{lem_ife_projection}, we only need to estimate $\| \bfPi^{+}_K \tilde{\bfu} - \bfJ^+_K \bfu \|_{0,K^-_h}$. Again, let us write 
\begin{equation}
\label{Pi_curl_est_eq4}
\bfPi^{+}_K \tilde{\bfu} - \bfJ^+_K \bfu  = \bfq + q_0 (x_2 - x^m_2, -(x_1-x^m_1))^{\intercal},
\end{equation}
where $q_0$ and $\bfq$ are piecewise scalar- and vector-valued constants. Next, we notice
$\bfq^+ = M\bfq^-$ and $\alpha^+q^+_0 = \alpha^-q^-_0$. Then, we have
\begin{equation}
\begin{split}
\label{Pi_curl_est_eq5}
\| \bfPi^{+}_K \tilde{\bfu} - \bfJ^+_K \bfu \|_{0,K^-_h} & \lesssim \| \bfq^+ \|_{0,K^-_h} + h_K\| q^+_0 \|_{0,K^-_h}  \lesssim \| M\bfq^- \|_{0,K^-_h} + h_K\| q^-_0 \|_{0,K^-_h} \\
& \lesssim \| \bfPi^{-}_K \bfu - \bfJ^-_K \bfu \|_{0,K^-_h} + h_K\| \curl( \bfPi^{-}_K \bfu - \bfJ^-_K \bfu) \|_{0,K^-_h},
\end{split}
\end{equation}
where in the last inequality we have inserted $q^-_0 (-(x_2 - x^m_2), x_1-x^m_1)^{\intercal}$ and used $q_0 = \curl( \bfPi_K \tilde{\bfu} - \bfJ_K \bfu)$. Now, inserting $\tilde{\bfu}^-=\bfu^-_E$ in the right-hand side of \eqref{Pi_curl_est_eq5}, applying Lemma \ref{quasi_curl}, and \eqref{Pi_curl_est_eq1} in Step 1 lead to the desired estimate \eqref{Pi_curl_est_eq3-1} of Step 2. This wraps up the case of $+$. Combing \eqref{utilde_eq01} and \eqref{Pi_curl_est_eq3-1} through the triangle inequality finishes the proof of \eqref{Pi_curl_est_eq01}.

\vspace{0.1in}
Step 3. As for \eqref{Pi_curl_est_eq02}, by Lemma \ref{curl_uI_proj} and \eqref{utilde_eq02}, we use that the projections are the best approximation to obtain
\begin{equation}
\begin{split}
\label{Pi_curl_est_eq6}
\| \sqrt{\alpha_h}( \curl\,\tilde{\bfu} - \curl \, \bfu_I) \|_{0,K} &  \le \| \sqrt{\alpha_h}( \curl\,\tilde{\bfu} - \curl \bfJ_K\bfu) \|_{0,K}   \lesssim h_K  \| \bfu \|_{E,\curl,1,\omega_K},
\end{split}
\end{equation}
where we have also applied Lemma \ref{quasi_curl}. Again, we have taken the advantage that both $\tilde{\bfu}$ and $\bfJ_K\bfu$ are piecewisely defined on $K$ separated by $\Gamma^K_h$. Then, similar to the argument above, it only remains to estimate
 $$
 \|  \curl\, (\bfJ^+_K\bfu - \, \bfu^+_I )\|_{0,K^-_h} \leqslant \frac{\alpha_h^-}{\alpha_h^+}  \|  \curl\, (\bfJ^-_K\bfu - \, \bfu^-_I )\|_{0,K^-_h}.
 $$ 
The right hand side above follows from inserting $\curl\,\bfu^-_E$ in between, and applying \eqref{Pi_curl_est_eq6} and Lemma \ref{quasi_curl} respectively on the two terms from the triangle inequality.
\end{proof}

Finally, the trace inequality also holds for $\bfH(\curl)$ IFE functions regardless of interface location.

\begin{lemma}[A trace inequality for $\bfH(\curl)$ IFE functions~\cite{2020GuoLinZou}]
\label{lem_trace_inequa}
For each interface element $K$ and its edge $e$, there holds
\begin{equation}
\label{lem_trace_inequa_eq0}
\| \bfv_h \|_{0,e} \lesssim h^{-1/2}_K \| \bfv_h \|_{0,K}, \quad\quad \forall \bfv_h\in \bfS^e_h(K).
\end{equation}
\end{lemma}

\section{$H^1$ elliptic interface problems}

In this section, we present the IVE method for solving the $H^1$-elliptic interface problem and give the optimal order convergence analysis. 

\subsection{Scheme}
\label{subsec:H1_IVE}
Define the local bilinear form on an \RG{interface element} $K$ as: $a^{n,K}_h(\cdot,\cdot): H^1(K)\times H^1(K) \rightarrow \mathbb{R}$ where
\begin{equation}
\label{eq:bilinear}
    a_h^{n,K}(u_h, v_h):=(\beta_h \nabla \Pi_K u_h, \nabla \Pi_K v_h)_K
    +S^n_K(u_h-\Pi_K u_h, v_h-\Pi_K v_h).
\end{equation}

One of the keys for VEM is the choice of the stabilization term. Here, following~\cite{2018CaoChen}, we consider the one associated with the $H^{1/2}(e)$ seminorm on $e\in \mathcal{E}_K$:
\begin{equation}
\label{half_semi_norm}
(w_h,z_h)_{1/2,e} := \int_e \int_e \beta_e  \frac{ (w_h( \bfx) - w_h( \bfy)) (z_h( \bfx) - z_h( \bfy)) }{ | \bfx - \bfy |^2 } \dd s(\bfx) \dd s(\bfy),
\end{equation}
where $\beta_e = \beta_h|_e$. Accordingly, $|\cdot|_{1/2,\mathcal{E}_K}$ is defined for any $w\in \Pi_{e\in \mathcal{E}_K} H^{1/2}(e)$ as 
\RG{\begin{equation}
\label{half_semi_norm_2}
|w|_{1/2,\mathcal{E}_K}^2 := \sum_{e\in\mathcal{E}_K} (w, w)_{1/2,e}.
\end{equation}}
Then, the stabilization term $S^n_K(\cdot, \cdot)$ is 
\begin{equation}
\label{eq:stab-h1}
    S^n_K(w_h,z_h) := \sum_{e\in\mathcal{E}_K} \beta_e(w_h, z_h)_{1/2,e} = \sum_{e \in \mathcal{E}_K}\beta_e (w_h(\bfb_e)-w_h(\bfa_e))(z_h(\bfb_e)-z_h(\bfa_e)).
\end{equation}
where the second identity is due to that both $w_h$ and $z_h$ are linear functions on each $e\in \mathcal{E}_K$. \LC{The difference-type stabilization in \eqref{eq:stab-h1} is first proposed in \cite{WriggersRustEtAl2016virtual}, and then analyzed in \cite{beirao2017stability}. Here we choose the discrete $1/2$ inner product  as the error analysis is robust to the edge length. For example, short edges are indeed unavoidable in our setting, of which the presence does not affect the robustness of the analysis.}
The proposed IVE scheme for solving \eqref{eq:problem-pde-interface} is to find $u_h\in V^n_h$ such that
\begin{equation}
\label{eq:problem-discretized}
    a^n_h(u_h, v_h):=\sum_{K\in \mathcal{T}_h} a_h^{n,K}(u_h, v_h)=\sum_{K\in \mathcal{T}_h}(f, \Pi_K v_h)_K, \quad  \forall v_h \in V^n_h,
\end{equation}
\RG{where the bilinear form on non-interface elements is simply the standard one $(\beta_h\nabla u_h, \nabla v_h)_K$. The well-posedness of the scheme above is given by Lemma \ref{lem_h1_norm} below.}
We define the energy norm 
\begin{equation}
\label{n_norm_h1}
\tnorm{v_h}^2_n := a^n_h(v_h,v_h).
\end{equation}
\begin{lemma}
\label{lem_h1_norm}
$\tnorm{\cdot}_n$ is a norm on $H^1_0(\Omega)\cap V^n_h$.
\end{lemma}
\begin{proof}
Suppose $\tnorm{v_h}_n = 0$ for some $v_h\in H^1_0(\Omega)\cap V^n_h$. on any $K\in \mathcal{T}^i_h$, by \eqref{eq:bilinear}, $\|\beta_h^{1/2}\nabla \Pi_K v_h\|_{0,K}=0$  implies $\Pi_K v_h\in \mathbb{P}_0(K)$. Moreover, $|({\rm I}-\Pi_K) v_h |_{1/2,e}=0$ implies $v_h \in \mathbb{P}_0(e)$ on each $e\in\mathcal{E}_K$. By $v_h\in C^0(\partial K)\cap H^1(K)$ in \eqref{virtual_space}, $v_h\in \mathbb{P}_0(K)$. The same result holds on non-interface elements trivially. Therefore, the continuity in \eqref{virtual_space_glob} and the boundary condition on $\partial \Omega$ lead to $v_h\equiv0$.
\end{proof}

\subsection{An Error Equation}

Given $u\in H^2(\beta,\mathcal{T}_h)$, 
since the global virtual element space $V_h^n$ is conforming, 
there always holds $u_I\in H^1(\Omega)$ given by \eqref{interp_UI}. Our analysis is based on the following error decomposition:
\begin{equation}
\label{error_decomp_h1}
\xi_h = u -u_I \quad\quad \text{and} \quad\quad \eta_h = u_I - u_h.
\end{equation}
The estimate of $\xi_h$ is from the interpolation error estimate and $\eta_h$ will be derived from an error equation. The IVE and IFE coincide with the standard simplicial finite element consisting only polynomials, thus the proposed stabilization vanishes. As a result, estimates on non-interface elements fall into the standard FEM regime; and our focus will be thus on the interface elements. 

We follow~\cite{2018CaoChen} to derive an error equation for $\eta_h = u_I - u_h$.

\begin{lemma}[Error equation]
\label{lem_err_eqn}
Let $u\in H^2(\beta;\mathcal T_h)$ be the solution to \eqref{eq:problem-pde-interface} and $u_h$ be the solution to \eqref{eq:problem-discretized}. Denote by $\eta_h=u_h-u_I$, then the following identity holds
\begin{equation}
\begin{aligned}
\tnorm{\eta_h}^2_n  = \sum_{K\in \mathcal{T}_h} &
\Bigl\{ (\beta_h \nabla \Pi_K (u-u_I), \nabla \Pi_K \eta_h)_K +(\beta_h \nabla (u-\Pi_K u)\cdot \mathbf{ n}, \eta_h-\Pi_K \eta_h)_{\partial K}\\
&-S_K^n(u_I-\Pi_K u_I,\eta_h-\Pi_K \eta_h) +((\beta-\beta_h) \nabla u, \nabla \Pi_K \eta_h)_K \Bigr\}.
\end{aligned}
\end{equation}
\end{lemma}
\begin{proof}
We start by the following
\begin{align}
\label{lem_err_eqn_1}
 \tnorm{\eta_h}^2_n ={}& a^n_h(u_h,\eta_h)-a^n_h(u_I,\eta_h)  
\\
=& \sum_{K\in \mathcal{T}_h}(f, \Pi_K \eta_h)_K-a^n_h(u_I,\eta_h) \tag{Problem \eqref{eq:problem-discretized}}
\\
=& \sum_{K\in \mathcal{T}_h}(-\nabla \cdot (\beta \nabla u), \Pi_K \eta_h)_K-a^n_h(u_I,\eta_h)
\tag{Original PDE}
\\
=& \sum_{K\in \mathcal{T}_h}\big[ \underbrace{(\beta \nabla u, \nabla \Pi_K \eta_h)_K}_{({\rm I})} - \underbrace{ (\beta \nabla u\cdot \mathbf{ n}, \Pi_K \eta_h)_{\partial K} }_{({\rm II})}\big]
\tag{Integration by parts}  -a^n_h(u_I,\eta_h) \nonumber.
\end{align}
In the last identity above, 
the flux jump conditions of $u$ \eqref{eq:jump} and the continuity of $\Pi_K \eta_h$ on $K$ are also used. For the term $({\rm I})$ in \eqref{lem_err_eqn_1}, using the definition of $\Pi_K$ we have
\begin{equation}
\begin{split}
\label{lem_err_eqn_2}
({\rm I}) &= (\beta_h \nabla u, \nabla \Pi_K \eta_h)_K + ((\beta-\beta_h) \nabla u, \nabla \Pi_K \eta_h)_K  \\
&= (\beta_h \nabla \Pi_K u, \nabla \Pi_K \eta_h)_K + ((\beta-\beta_h) \nabla u, \nabla \Pi_K \eta_h)_K.
\end{split}
\end{equation}
For the term $({\rm II})$, since $\beta = \beta_h$ on $\partial K$, we obtain
\begin{equation}
\label{lem_err_eqn_3}
\sum_{K\in \mathcal{T}_h}({\rm II}) = \sum_{K\in \mathcal{T}_h} (\beta_h \nabla u\cdot \mathbf{ n}, \Pi_K \eta_h)_{\partial K} = \sum_{K\in \mathcal{T}_h} (\beta_h \nabla u\cdot \mathbf{ n}, \Pi_K \eta_h - \eta_h)_{\partial K},
\end{equation}
where in the second identity we have used $\eta_h=u_h-u_I$ being continuous across each edge as it is in the virtual element space $V^n_h$. 
Using integration by parts on the subelements $K^{\pm}_h$, the flux jump conditions of the IFE functions on $\Gamma^K_h$, $\eta_h-\Pi_K \eta_h$ being continuous across $\Gamma^K_h$, and definition of the projection $\Pi_K$, we have
\begin{equation}
\label{lem_err_eqn_3_1}
\begin{split}
(\beta_h \nabla \Pi_K u\cdot \mathbf{ n}, \eta_h-\Pi_K \eta_h)_{\partial K} & = \sum_{s=\pm} (\beta_h \nabla \Pi_K u\cdot \mathbf{ n}, \eta_h-\Pi_K \eta_h)_{\partial K^s_h}  \\
& = \sum_{s=\pm} (\beta_h \nabla \Pi_K u , \nabla( \eta_h-\Pi_K \eta_h)  )_{K_h^s}  = 0.
\end{split}
\end{equation}
Thus, \eqref{lem_err_eqn_3} further becomes
\begin{equation}
\label{lem_err_eqn_4}
\sum_{K\in \mathcal{T}_h} ({\rm II})  = \sum_{K\in \mathcal{T}_h} (\beta_h \nabla (u - \Pi_K u ) \cdot \mathbf{ n}, \Pi_K \eta_h - \eta_h)_{\partial K}.
\end{equation}
Putting \eqref{lem_err_eqn_2} and \eqref{lem_err_eqn_4} into \eqref{lem_err_eqn_1}, and using the formula of $a^n_h(u_I,v_h)$, we obtain the desired result.
\end{proof}

In the derivation above, there are two steps involving integration by parts: the one in \eqref{lem_err_eqn_1} is for the exact solution $u$ with respect to the subelements $K^{\pm}$, and another one in \eqref{lem_err_eqn_3_1} is for IVE and IFE functions with respect to the subelements $K^{\pm}_h$. Their difference corresponds to their respective jump conditions imposed on $\Gamma$ or $\Gamma^K_h$, such that those extra terms occurring on $\Gamma$ or $\Gamma^K_h$ can be cancelled.

\subsection{Error Estimates}
In this section, we proceed to estimate the solution errors. Based on the error equation in Lemma \ref{lem_err_eqn}, we first get an error bound for $u_h-u_I$.

\begin{theorem}[A priori error bound]
\label{thm_err_bound}
Let $u\in H^2(\beta;\mathcal T_h)$ be the solution to \eqref{eq:problem-pde-interface} and $u_h$ be the solution to \eqref{eq:problem-discretized}. Denote by $\eta_h=u_h-u_I$. Then there holds
\begin{align}
\label{eq:err-bound-term-1}
 \tnorm{\eta_h}_n  \lesssim \sum_{K\in \mathcal{T}_h}& \Big [ \|\beta_h^{1/2}\nabla \Pi_K (u-u_I)\|_{0,K}
+ h^{1/2}_K\|\beta_h^{1/2}\nabla (u-\Pi_K u)\cdot \mathbf{ n}\|_{0,\partial K} \\
&+ |\beta_h^{1/2}(u_I-\Pi_K u_I) |_{1/2,\mathcal{E}_K}
+\|\beta_{\max}^{1/2} \nabla u\|_{0,\delta{K}} \Big ].
\end{align}
\end{theorem}
\begin{proof}
Note that $\beta\neq \beta_h$ only on $\delta{K}$, thus for the error equation in Lemma \ref{lem_err_eqn}, applying the Cauchy-Schwarz inequality, we have
\begin{equation}
\label{thm_err_bound_eq1}
\begin{aligned}
 \tnorm{\eta_h}^2_n  
\leq& \sum_{K\in \mathcal{T}_h} \Big( \|\beta_h^{1/2}\nabla \Pi_K (u-u_I)\|_{0,K}\| \beta_h^{1/2}\nabla \Pi_K \eta_h\|_{0,K}\\&
+\|\beta_h^{1/2}\nabla (u-\Pi_K u)\cdot \mathbf{ n}\|_{0,\partial K}\|\beta_h^{1/2}(\eta_h-\Pi_K \eta_h)\|_{0,\partial K}\\
& + |\beta_h^{1/2}(u_I-\Pi_K u_I) |_{1/2,\mathcal{E}_K} |\beta_h^{1/2}(\eta_h-\Pi_K \eta_h) |_{1/2,\mathcal{E}_K}\\
&+\|\beta_{\max}^{1/2} \nabla u\|_{0,\delta{K}}\|\beta_{\max}^{1/2}\nabla \Pi_K \eta_h\|_{0,K} \Big).
\end{aligned}
\end{equation}
In the bound above, it is clear that $\| \beta_h^{1/2}\nabla \Pi_K \eta_h\|_{0,K}$ and $|\beta_h^{1/2}(\eta_h-\Pi_K \eta_h)|_{1/2,\mathcal{E}_K}$ are bounded above by $\tnorm{\eta_h}_n$, and $\| \beta_{\max}^{1/2}\nabla \Pi_K \eta_h\|_{0,K}$ is also bounded above by $\tnorm{\eta_h}_n$ with a $\beta$ dependent constant. 

To estimate the remaining second term in \eqref{thm_err_bound_eq1}, 
we note that $\int_{\partial K} (\eta_h-\Pi_K \eta_h)\dd s =0$, 
thus applying \eqref{eq:Poincare-boundary} edge-wise in Theorem \ref{Poincare} yields
$$
\|\beta_h^{1/2}(\eta_h-\Pi_K \eta_h)\|_{0,\partial K}\lesssim h_K^{1/2} |\beta_h^{1/2}(\eta_h-\Pi_K \eta_h)|_{1/2,\mathcal{E}_K} \lesssim h^{1/2}_K \tnorm{\eta_h}_n.
$$
Combining the estimates above and cancelling out a $\tnorm{\eta_h}_n$ on each side, we get the desired a priori estimate.
\end{proof}

To get the optimal order of convergence of the proposed method, our task is to estimate each term \RG{on the} right-hand side of the error bound \eqref{eq:err-bound-term-1}. Before getting into the estimate, we emphasize that the set $\mathcal{E}_K$ consists of the edges formed by element vertices and cut points.
Therefore, to avoid confusion in the following discussion, for each edge $e\in\mathcal{E}_K$ that connects an element vertex and a cut point, we will use $\hat{e}$ to denote the edge containing $e$ on the triangle in the background mesh \RG{(e.g. $e=\overline{\bfa_1\bfb_1}$ to $\hat{e}=\overline{\bfa_1\bfa_2}$ in Figure \ref{fig:interface-single})}. Now, let us first derive the estimate of the first term in the right-hand side of the error bound in \eqref{eq:err-bound-term-1}.

\begin{lemma}
\label{projection term estimate}
Let $u\in H^2(\beta;\mathcal T_h)$, then on any $K\in \mathcal{T}^i_h$ there holds
\begin{equation}
\label{projection term estimate eq1}
\|\beta_h^{1/2}\nabla \Pi_K (u-u_I)\|_{0,K}
\lesssim h_K \|u \|_{E, 2,\omega_K}.
\end{equation}
\end{lemma}
\begin{proof}
By the definition of projection, we immediately have
\begin{equation*}
\begin{aligned}
 \|\beta_h^{1/2}\nabla \Pi_K (u-u_I)\|_{0,K}^2 =& (\beta_h \nabla \Pi_K (u-u_I),\nabla \Pi_K (u-u_I))_K
 = (\beta_h \nabla \Pi_K (u-u_I),\nabla (u-u_I))_K.
\end{aligned}
\end{equation*}
Using integration by parts on the subelements $K^{\pm}_h$,  $\Pi_K(u-u_I)$ satisfying the jump condition on $\Gamma_K$, and $u-u_I \in H^1(K)$, we have 
\begin{equation}
\label{projection term estimate eq2}
\begin{aligned}
 \|\beta_h^{1/2}\nabla \Pi_K (u-u_I)\|_{0,K}^2 =& (\beta_h \nabla \Pi_K (u-u_I)\cdot \mathbf{ n},u-u_I)_{\partial K}\\
 \leq& \|\beta_h^{1/2} \nabla \Pi_K (u-u_I)\cdot \mathbf{ n}\|_{0,\partial K}\|\beta_h^{1/2}(u-u_I)\|_{0,\partial K}.
\end{aligned}
\end{equation}
For each edge on $\partial{K}$, applying the IFE trace inequality in \RG{Theorem \ref{IFE trace}}, we obtain
\begin{equation}
\label{projection term estimate eq3}
\begin{split}
\|\beta_h^{1/2} \nabla \Pi_K (u-u_I)\cdot \mathbf{ n}\|_{0,e} & \leq \|\beta_h^{1/2} \nabla \Pi_K (u-u_I)\cdot \mathbf{ n}\|_{0,\hat{e}} \lesssim h_K^{-1/2} \|\beta_h^{1/2}\nabla \Pi_K (u-u_I)\|_{0,K}.
 \end{split}
\end{equation}
Putting \eqref{projection term estimate eq3} into \eqref{projection term estimate eq2} and cancelling out the term $ \|\beta_h^{1/2}\nabla \Pi_K (u-u_I)\|_{0,K}$ leads to
\begin{equation}
\label{projection term estimate eq4}
\|\beta_h^{1/2}\nabla \Pi_K (u-u_I)\|_{0,K}\lesssim h_K^{-1/2}\|\beta_h^{1/2}(u-u_I)\|_{0,\partial K}.
\end{equation}
So it remains to estimate the right-hand side above. Notice $\beta_h$ is constant on each edge $e\in \mathcal{E}_K$. Without loss of generality, consider an $e\subset \partial K^+$, by the interpolation estimate on this edge, we have
\begin{equation}
\begin{split}
\label{projection term estimate eq5}
\| \beta_h^{1/2}(u-u_I) \|_{0,e} & \lesssim h^{3/2}_e|u|_{3/2,e}  \lesssim h^{3/2}_K |u^+_E |_{3/2,\hat{e}}  \lesssim h_K^{3/2} |u^{+}_E|_{2,K}
\end{split}
\end{equation}
where in the last inequality, we have also applied the trace inequality in~\cite[Lemma 6.2]{2018CaoChen} on $\nabla u_E^+|_{\hat{e}}$. Putting \eqref{projection term estimate eq5} into \eqref{projection term estimate eq4} gives the desired estimate on this edge. Similar arguments apply to the case $e\subset \partial K^{-}$ which together finishes the proof.
\end{proof}

The estimate of the second and third terms in the right-hand side of the error bound \eqref{eq:err-bound-term-1} relies on the estimate of every polynomial component of $\Pi_K^{\pm}$ on the whole element $K$ which has been established in Lemma \ref{lem_ife_projection}.

\begin{lemma}
\label{projection error on edge}
Let $u\in H^2(\beta;\mathcal T_h)$, then on any $K\in \mathcal{T}^i_h$ there holds
\begin{equation}
\begin{split}
    \|\beta_h^{1/2}\nabla (u-\Pi_K u)\cdot \mathbf{ n}\|_{0,\partial K} \lesssim 
h^{1/2}_K \|u \|_{E,2,\omega_K}  + h^{-1/2}_K |u|_{E,1,\delta{K}} .
 \end{split}
\end{equation}
\end{lemma}
\begin{proof}
Without loss of generality, we only consider $+$ side. Given an edge $e\in \mathcal{E}_K$ with $e\subseteq K^+_h$ and its extension $\hat{e}$ as an edge of $K$, we apply the trace inequality to obtain
\begin{equation*}
\begin{split}
\| \beta_h \nabla (u-\Pi_K u) \cdot \mathbf{ n} \|_{0,e} &\leq (\beta^{+})^{1/2}\|\nabla (u^{+}_E -\Pi_K^{+} u) \cdot \mathbf{ n} \|_{0,\hat{e}} \\
 & \lesssim h^{-1/2}_K | u^{+}_E -\Pi_K^{+} u |_{1,K} + h^{1/2}_K | u^{+}_E |_{2,K}
\end{split}
\end{equation*}
which yields the desired result by Lemma \ref{lem_ife_projection}.
\end{proof}

\begin{lemma}
\label{projection_interpolation error on edge}
Let $u\in H^2(\beta;\mathcal T_h)$, then on any $K\in \mathcal{T}^i_h$ there holds 
\begin{equation}
    |\beta_h^{1/2}(u_I-\Pi_K u_I)|_{1/2,\mathcal{E}_K} \lesssim 
h_K \|u \|_{E, 2,\omega_K}    + |u|_{E, 1,\delta{K}}.
\end{equation}
\end{lemma}
\begin{proof}
\RG{Recall that $|\cdot|_{1/2,\mathcal{E}_K}$ is defined in \eqref{half_semi_norm_2}. It suffices to establish an edge-wise estimate under $|\cdot|_{1/2,e}$ of which the definition is given in \eqref{half_semi_norm}}. For each edge, since $\beta_h$ is a constant, 
\[
| \beta^{1/2}_h( u_I-\Pi_K u_I) |_{1/2,e}
\lesssim 
\underbrace{|u_I-\Pi_K u|_{1/2,e}}_{({\rm I})} + \underbrace{|\Pi_K(u-u_I)|_{1/2,e}}_{({\rm II})}.
\]
In the following discussion, without loss of generality we only consider $e\subseteq K^+_h$. For $({\rm I})$, since $u_I-\Pi_K u$ is linear on $e$, and $u$ and $u_I$ match at the end points $\bfa_e$ and $\bfb_e$ of $e$, we obtain
\begin{equation}
\label{projection_interpolation error on edge eq2}
\begin{aligned}
({\rm I}) = \; & \snorm{ (u_I-\Pi_K^{+} u) |_{\bfa_e}^{\bfb_e} }
= \snorm{ (u-\Pi_K^{+} u) |_{\bfa_e}^{\bfb_e} }
= \left|\int_e\partial_e(u-\Pi_K^{+} u) \dd s\right| 
\leq \; h_e^{1/2}|u-\Pi_K^{+} u|_{1,e}.
\end{aligned}
\end{equation}
Replacing $u$ by its extension $u^+_E$ and recalling that $\Pi^+_Ku$ is a polynomial being trivially used on the whole element $K$, we apply the standard trace inequality and Lemma \ref{lem_ife_projection} to get
\begin{equation}
\begin{split}
\label{projection_interpolation error on edge eq1}
  ({\rm I})  & \leq h^{1/2}_K |u_E^{+}-\Pi_K^{+} u|_{1,\hat{e}}\lesssim  |u_E^{+}-\Pi_K^{+} u|_{1,K} + h_K |u^+_E|_{2,K}  \lesssim h_K \|u^{\pm}_E \|_{2,\omega_K}    + |u^{\pm}_E|_{1,\delta{K}}.
  \end{split}
\end{equation}

For $({\rm II})$, applying the trace inequality for IFE functions in Theorem \ref{IFE trace}, and Lemma \ref{projection term estimate}, we obtain
\begin{equation}
\label{projection_interpolation error on edge eq3}
\begin{aligned}
({\rm II}) = & \snorm{ \Pi_K(u - u_I )|_{\bfa_e}^{\bfb_e} } = \snorm{ \int_e \partial_e\Pi_K(u - u_I ) \dd s } \\
 \leq &  h_e^{1/2}|\Pi_K(u-u_I)|_{1,\hat{e}}
\lesssim  h_K^{-1/2}h_e^{1/2}|\Pi_K(u-u_I)|_{1,K}
\leq h_K\|u^{\pm}_E \|_{2,\omega_K}.
\end{aligned}
\end{equation}
Combining the estimates of $({\rm I})$ and $({\rm II})$, we have the desired result.
\end{proof}

Combining the results of Lemma \ref{projection term estimate}, \ref{projection error on edge} and \ref{projection_interpolation error on edge} and the error bound in Theorem \ref{thm_err_bound}, we achieve the following conclusion.

\begin{theorem}
\label{u-u_h_estimate}
Let $u\in H^2(\beta;\mathcal T_h)$ be the solution to \eqref{eq:problem-pde-interface} and $u_h$ be the solution to \eqref{eq:problem-discretized}, we have
\begin{equation}
    \tnorm{u-u_h}_n\lesssim h \|u \|_{2,\cup \, \Omega^{\pm}}.
\end{equation}
\end{theorem}
\begin{proof} 
\RG{The triangle inequality} yields $\tnorm{u-u_h}_n \leq \tnorm{u-u_I}_n + \tnorm{u_I-u_h}_n$.
For $\tnorm{u_I-u_h}_n$, combining the results of Lemmas \ref{projection term estimate}, \ref{projection error on edge} and \ref{projection_interpolation error on edge} and the error bound in Theorem \ref{thm_err_bound}, we have
\begin{equation}
\begin{split}
\label{u-u_h_estimate_eq1}
\tnorm{u_I-u_h}_n & \lesssim \sum_{K\in\mathcal{T}^n_h} h_K \| u \|_{2,K} + \sum_{K\in\mathcal{T}^i_h} \left( h_K \| u \|_{E, 2,\omega_K} + |u|_{E, 1,\delta{K}} \right) \lesssim h \| u \|_{E, 2,\Omega} \lesssim h \|u \|_{2,\cup \, \Omega^{\pm}},
\end{split}
\end{equation}
where we have used the finite overlapping property of $\omega_K$ and the strip argument in Lemma \ref{strip region} to control $|u|_{1,\delta{K}}$ and finally the boundedness for Sobolev extensions.

Then we proceed to estimate $\tnorm{u-u_I}_n$. Since it is trivial on non-interface elements, we only need to estimate it on interface elements. By the triangle inequality, we have
\begin{equation}
\begin{split}
\label{u-u_h_estimate_eq2}
\tnorm{u-u_I}_n \lesssim & \sum_{K\in\mathcal{T}^i_h} \Big(\|\beta_h^{1/2} \nabla \Pi_K (u-u_I)\|_{0,K} 
+ | u - u_I |_{1/2,\mathcal{E}_K} \Big) +  \sum_{K\in\mathcal{T}^n_h} h_K \| u \|_{2,K}.
\end{split}
\end{equation}
The first term can be handled by Lemma \ref{projection term estimate}. For the second term, given $e\in\mathcal{E}_K$ and without loss of generality assuming it is $K^+_h$, by the interpolation estimate in 1D and the \RG{trace inequality \cite[Lemma 6.2]{2018CaoChen}}, we have
\begin{equation}
\label{u-u_h_estimate_eq3}
| u - u_I |_{1/2,e} \lesssim h_e |u|_{3/2,e} \lesssim h_e |u^+_E|_{3/2,\hat{e}}  \lesssim h_K \| u^+_E \|_{2,K}
\end{equation}
where $\hat{e}$ is the extension of $e$. Putting \eqref{u-u_h_estimate_eq3} to \eqref{u-u_h_estimate_eq2} and applying the boundedness for Sobolev extensions, we have the desired result.
\end{proof}

\section{$\bfH(\text{curl})$ Interface Problems}

In this section, we present an IVEM for the $\bfH(\curl)$-elliptic interface problem and give an optimal order error estimate. 

\subsection{Scheme}
We first present the scheme for the $\bfH(\curl)$ interface problem. Define the local discrete bilinear form on an interface element $K$ as: $a^{e,K}_h(\cdot,\cdot): \bfH(\curl;K)\times \bfH(\curl;K) \rightarrow \mathbb{R}$ where
\begin{equation}
\begin{aligned}
\label{loc_bilinear_curl}
    a_h^{e,K}(\bfu_h, \bfv_h) := {}& (\alpha_h \curl \bfu_h, \curl \bfv_h)_K 
    + (\beta_h\bfPi_K \bfu_h, \bfPi_K \bfv_h )_K    + S^e_K( \bfu_h-\bfPi_K \bfu_h, \bfv_h-\bfPi_K \bfv_h).
\end{aligned}
\end{equation}
\RG{Similarly, $\bfPi_K$ reduces to an identity operator on non-interface elements, and thus the local bilinear forms do not contain any projection or stabilization terms.} Following~\cite{2021CaoChenGuo}, using the same $\beta_e$ in \eqref{eq:stab-h1}, we directly employ the DoFs to construct the stabilization $S^e_K(\cdot, \cdot)$ :
\begin{equation}
\label{stab_curl}
S^e_K(\bfw_h,\bfz_h) := \sum_{e\in\mathcal{E}_K}  \beta_e ( \bfw_h \cdot \bft, \bfz_h \cdot\bft)_{0,e}.
\end{equation}
With these preparations, the IVEM for solving \eqref{inter_PDE} is to find $\bfu_h \in \bfV^e_h$ such that
\begin{equation}
\label{IVEM_curl}
a^e_h(\bfu_h, \bfv_h) := \sum_{K\in\mathcal{T}_h} a^{e,K}_h(\bfu_h,\bfv_h) = \sum_{K\in\mathcal{T}_h}(\bff, \bfPi_K\bfv_h)_{K}, \quad \forall \bfv_h \in \bfV^e_h,
\end{equation}
\RG{where the local bilinear form on non-interface elements is the standard one $(\alpha_h\curl\bfu_h, \curl\bfv_h) + (\beta_h\bfu_h,\bfv_h)$.}

\RG{
\begin{remark}
Note that the scaling in \eqref{stab_curl} is different from the conventional VEM using $h$~\cite{2021CaoChenGuo,2020VeigaDassiMascotto,2020BeiroMascotto} (or $h^{1/2}$ on the boundary terms in the induced norm). In this work, the proposed stabilization term above is larger than the one with the $h$ weight, yet this will not downgrade the coercivity constant to become mesh size dependent, see Lemma \ref{lem_curl_normequ} below. \LC{The consistency error may consequently become bigger. However, since the stabilization is only needed near the interface, the overall consistency error is still of the optimal order.} We postpone the detailed mathematical reasoning to Remark \ref{rem_scaling}. Here we emphasize that the constant weight stabilization is one of the keys to ensure the optimal order of convergence, see Lemma \ref{lem_curl_est_piu_bd} and Remark \ref{rem_scaling}. 
\end{remark}
}

\subsection{Coercivity}
We begin with defining an energy norm:
\begin{equation}
\label{e_norm_curl}
\tnorm{\bfv_h}^2_e : = a^e_h(\bfv_h,\bfv_h).
\end{equation}
We first show the quantity in \eqref{e_norm_curl} is indeed a norm.
\begin{lemma}
\label{lem_curl_bound}
Given $\bfv_h \in \bfV^e_h(K)$, there holds 
\begin{equation}
\label{lem_curl_bound_eq0}
\| \bfv_h \|_{0,K} \lesssim \frac{\beta_{\max}}{\beta_{\min}}  \left( h_K \| \curl \, \bfv_h \|_{0,K} + h^{1/2}_K\sum_{e\in\mathcal{E}_K} \| \bfv_h\cdot \bft \|_{0,e} \right).
\end{equation} 
\end{lemma}
\begin{proof}
Given each $\bfv_h \in \bfV^e_h(K)$, let $\varphi_h$ be the corresponding function in Remark \ref{lem_varphi_stab}. 
Then, $-\nabla\cdot(\beta^{-1}_h \nabla \varphi) = \curl\,\bfv_h$ and $ \beta^{-1}_h \nabla\varphi\cdot\bfn =  -\bfv_h\cdot\bft$ on $\partial K$. Using integration by parts, we obtain
\begin{equation}
\begin{split}
\label{lem_curl_bound_eq2}
\| \bfv_h \|^2_{0,K} & = \int_K \beta^{-1}_h \bcurl \varphi_h \cdot \beta^{-1}_h \bcurl \varphi_h \dd\bfx  \lesssim \beta_{\min}^{-1}  \int_K \beta^{-1}_h \nabla \varphi_h \cdot \nabla \varphi_h \dd\bfx  \\
& =  \beta_{\min}^{-1} \Big( - \int_K \varphi_h \nabla\cdot( \beta^{-1}_h  \nabla \varphi_h )\dd\bfx + \int_{\partial K} \varphi_h \beta^{-1}_h\nabla \varphi_h \cdot\bfn \dd s \Big) \\
& \lesssim \beta_{\min}^{-1} \Big( \| \varphi_h \|_{0,K} \| \curl \, \bfv_h \|_{0,K} + \| \varphi_h \|_{0,\partial K} \|\bfv_h\cdot\bft \|_{0,\partial K} \Big).
\end{split}
\end{equation} 
Applying \eqref{varphi_stab} and cancelling one term of $\| \bfv_h \|_{0,K}$ leads to the desired result.
\end{proof}
We highlight that the hidden constant in Lemma \ref{lem_curl_bound} is still independent of the interface location. But, compared with Proposition 4.1 of \cite{2020BeiroMascotto}, our result involves the extra term $h_K \| \curl \, \bfv_h \|_{0,K}$. It yields the following coercivity.
\begin{lemma}
\label{lem_curl_normequ}
For all $\bfv_h\in\bfV^e_h$, there holds 
\begin{equation}
\label{lem_curl_normequ_eq0}
\| \bfv_h \|_{ \bfH(\curl;\Omega) } \lesssim \tnorm{ \bfv_h}_e.
\end{equation}
\end{lemma}
\begin{proof}
As the norm induced by $a_h^{e,K}(\cdot, \cdot)$ agrees with $\|\cdot\|_{\bfH(\curl;\Omega)}$ on non-interface elements, it suffices to establish the estimates on an interface element $K$. The triangle inequality directly yields
\begin{equation}
\label{lem_curl_normequ_eq1}
\| \bfv_h \|_{0,K} \le \| \bfPi_K \bfv_h \|_{0,K} + \| \bfv_h - \bfPi_K \bfv_h \|_{0,K}.
\end{equation}
We note that $\bfPi_K\bfv_h\in \bfV^e_h(K)$, then it follows from Lemma \ref{lem_curl_bound} and $h_K\lesssim \mathcal{O}(1)$ that
\begin{equation}
\begin{split}
\label{lem_curl_normequ_eq2}
 \| \bfv_h - \bfPi_K \bfv_h \|_{0,K}\lesssim h_K \| \curl \, \bfv_h \|_{0,K} + h^{1/2}_K\sum_{e\in\mathcal{E}_K} \| (\bfv_h - \bfPi_K \bfv_h)\cdot \bft \|_{0,e}.
 \end{split}
\end{equation}
Summing up \eqref{lem_curl_normequ_eq1} and \eqref{lem_curl_normequ_eq2} on all elements yields the desired result.
\end{proof}

\RG{\begin{remark}
\label{rem_coer}
In particular, \eqref{lem_curl_normequ_eq0} implies the coercivity of the bilinear form $a^{e,K}_h(\cdot,\cdot)$, and thus guarantees the existence and uniqueness of the solution to \eqref{IVEM_curl}. Comparing \eqref{lem_curl_normequ_eq2} and the stabilization term \eqref{stab_curl}, we see that such coercivity still holds independent of the mesh size as the applied stabilization is stronger ($\mathcal O(1)$ v.s. $\mathcal O(h_K^{1/2})$).
\end{remark}}

\subsection{An Error Equation}

Similar to the $H^1$ case, the analysis is based on the following error decomposition:
\begin{equation}
\label{error_decomp_hcurl}
\bfxi_h = \bfu - \bfu_I \quad\quad \text{and} \quad\quad \bfeta_h = \bfu_h - \bfu_I,
\end{equation}
where $\bfu_I$ is given by \eqref{interp_UI}. Let us present the error equation and error bounds.
\begin{lemma}[Error equation]
\label{lem_err_curl_eq0}
Let $\bfu\in \bfH^1(\curl,\alpha,\beta;\mathcal T_h)$ be the solution to \eqref{inter_PDE} and $\bfu_h$ be the solution to \eqref{IVEM_curl}. Then the following identity holds
\begin{equation}
\label{eq:err-eq-curl}
\begin{aligned}
\tnorm{\bfeta_h}^2_e  = \sum_{K\in\mathcal{T}_h} & \Big\{ \int_{\partial K} \alpha_h ( \curl\, \bfu - \curl\, \bfu_I)(\bfeta_h\cdot \bft - \bfPi_K\bfeta_h\cdot\bft ) \dd s + ( (\beta-\beta_h)\bfu, \bfPi_K\bfeta_h )_K \\
&+ ( \beta_h(\bfu - \bfPi_K\bfu_I), \bfPi_K\bfeta_h )_K - S^e_K(\bfu-\bfPi_K\bfu_I, \bfeta_h-\bfPi_K\bfeta_h) \Big\}.
\end{aligned}
\end{equation}
\end{lemma}
\begin{proof}
We proceed similarly as \eqref{lem_err_eqn_1} in Lemma \ref{lem_err_eqn}. Using the discretized problem \eqref{IVEM_curl}, the original PDE \eqref{inter_PDE}, and integration by parts elementwisely, we have
\begin{equation}
\begin{split}
\label{lem_err_curl_eq1}
\tnorm{\bfeta_h}^2_e & =  a^e_h(\bfu_h , \bfeta_h) - a^e_h(\bfu_I , \bfeta_h) 
\\
& = \sum_{K\in\mathcal{T}_h} (\bff, \bfPi_K \bfeta_h)_K 
- (\alpha_h \curl \, \bfu_I , \curl \, \bfeta_h)_K 
- ( \beta_h\bfPi_K\bfu_I, \bfPi_K\bfeta_h )_K
- S^e_K(\bfu_I-\bfPi_K\bfu_I, \bfeta_h - \bfPi_K\bfeta_h) 
\\
& =  \sum_{K\in\mathcal{T}_h} 
\underbrace{ (\bcurl \, \alpha \curl \,\bfu, \bfPi_K \bfeta_h)_K }_{(\rm Ia)} 
-  \underbrace{ (\alpha_h \curl \, \bfu_I , \curl \, \bfeta_h)_K }_{({\rm Ib}) } 
\\
& \quad\quad\quad +  \underbrace{ (\beta \bfu, \bfPi_K \bfeta_h) }_{({\rm IIa})} 
-  \underbrace{ ( \beta_h\bfPi_K\bfu_I, \bfPi_K\bfeta_h )_K}_{({\rm IIb})}  
- S^e_K(\bfu_I-\bfPi_K\bfu_I, \bfeta_h - \bfPi_K\bfeta_h).
\end{split}
\end{equation}
For $({\rm Ia})$, integration by parts and the continuity conditions for $\curl \bfu \in \widetilde{H}^1(\alpha,\mathcal{T}_h)$ and $\bfeta_h\in \bfH(\curl;\Omega)$ imply 
\begin{equation}
\label{lem_err_curl_eq2}
\begin{split}
\sum_{K\in\mathcal{T}_h}({\rm Ia}) &= - \sum_{K\in\mathcal{T}_h} \int_{\partial K} \alpha \curl\, \bfu \, (\bfPi_K\bfeta_h\cdot\bft) \dd s =  \sum_{K\in\mathcal{T}_h} \int_{\partial K} \alpha \curl\, \bfu \, ( \bfeta_h\cdot\bft - \bfPi_K\bfeta_h\cdot\bft ) \dd s.
\end{split}
\end{equation}
In addition, since $\alpha_h \curl\, \bfu_I$ is a constant and $\bfPi_K\bfeta_h\in \nabla S^n_h(K)$ by the exact sequence \eqref{thm_DR_3_curl}, we obtain
\begin{equation}
\label{lem_err_curl_eq3}
 \int_{\partial K} \alpha_h \curl\, \bfu_I \, (\bfPi_K\bfeta_h\cdot\bft) \dd s = \alpha_h \curl\, \bfu_I  \int_{\partial K} \bfPi_K\bfeta_h\cdot\bft \dd s = 0.
\end{equation}
So, using integration by parts again together with \eqref{lem_err_curl_eq3}, we have
\begin{equation}
\label{lem_err_curl_eq4}
({\rm Ib}) = \int_{\partial K} \alpha_h \curl \, \bfu_I (\bfeta_h\cdot \bft) \dd s = \int_{\partial K} \alpha_h \curl \, \bfu_I (\bfeta_h\cdot \bft - \bfPi_K\bfeta_h\cdot\bft ) \dd s.
\end{equation}
As $\alpha$ matches $\alpha_h$ on $\partial K$, we obtain
\begin{equation}
\label{lem_err_curl_eq5}
\sum_{K\in\mathcal{T}_h} ({\rm Ia}) +({\rm Ib})  = \sum_{K\in\mathcal{T}_h} \int_{\partial K} \alpha_h ( \curl\, \bfu - \curl\, \bfu_I)(\bfeta_h\cdot \bft - \bfPi_K\bfeta_h\cdot\bft ) \dd s.
\end{equation}
For the terms $({\rm II})$, we simply have
\begin{equation}
\label{lem_err_curl_eq6}
({\rm IIa}) - ({\rm IIb}) = ( (\beta-\beta_h)\bfu, \bfPi_K\bfeta_h )_K + ( \beta_h(\bfu - \bfPi_K\bfu_I), \bfPi_K\bfeta_h )_K.
\end{equation}
As for the stabilization term, using the fact that $(\bfeta_h - \bfPi_K\bfeta_h)\cdot\bft=:c$ is a constant on $e$, applying the definition of the interpolation $\int_e (\bfu_I\cdot\bft) c \dd s = \int_e (\bfu\cdot\bft) c \dd s$ yields the desired result.
\end{proof}

With the error equation above, we are able to derive the error bound for $\bfeta_h$.
\begin{theorem}[A priori error bound]
\label{thm_err_bound_curl}
Let $\bfu\in \bfH^1(\curl,\alpha,\beta;\mathcal T_h)$ be the solution to \eqref{inter_PDE} and $\bfu_h$ be the solution to \eqref{IVEM_curl}. Then it follows that
\begin{equation}
\label{thm_err_bound_curl_eq0}
\begin{aligned}
 \tnorm{\bfeta_h}_e \lesssim  \sum_{K\in\mathcal{T}_h} & \Big (\| \alpha_h ( \curl\, \bfu - \curl\, \bfu_I) \|_{0,\partial K} + \| \sqrt{\beta_h}\,(\bfu - \bfPi_K\bfu_I) \|_{0,K}  + \| \sqrt{\beta_h}\,(\bfu -\bfPi_K\bfu_I) \|_{0,\partial K} \Big )+ h \| \bfu \|_{1,\Omega}.
\end{aligned}
\end{equation}
\end{theorem}
\begin{proof}
It directly follows from the Cauchy-Schwarz inequality and the definition of stabilization $S^e_K(\cdot, \cdot)$, where the last term is due to Lemma \ref{strip region}.
\end{proof}

\subsection{Convergence Analysis}

We proceed to estimate each term in \eqref{thm_err_bound_curl_eq0}.

\begin{lemma}
\label{lem_est_curlu}
Let $\bfu\in \bfH^1(\curl,\alpha,\beta;\mathcal T_h)$. Then it follows that
\begin{equation}
\label{lem_est_curlu_eq0}
\| \alpha_h ( \curl\, \bfu - \curl\, \bfu_I) \|_{0,\partial K} \lesssim h^{1/2}_K \| \bfu \|_{E,\curl,1,\omega_K} .
\end{equation}
\end{lemma}
\begin{proof}
Since $K$ is shape regular, given an edge $e\in \mathcal{E}_K$, suppose $e\subset \partial K^+$ without loss of generality, then applying the trace inequality in Lemma \ref{H1 trace} for extensions on the whole $K$ yields
\begin{equation}
\label{lem_est_curlu_eq1}
\|  \curl\, \bfu - \curl\, \bfu_I  \|_{0,e} \lesssim h^{-1/2}_K \|  \curl\, \bfu^+_E - \curl\, \bfu^+_I  \|_{0,K} + h^{1/2}_K |  \curl\, \bfu^+_E |_{1,K} 
\end{equation}
which yields the desired result by \eqref{Pi_curl_est_eq02} in Lemma \ref{Pi_curl_est}.
\end{proof}

In order to estimate the rest terms of \eqref{thm_err_bound_curl_eq0}, we need the following result.
\begin{lemma}
\label{lem_est_Picurlu}
Let $\bfu\in \bfH^1(\curl,\alpha,\beta;\mathcal T_h)$. Then it follows that
\begin{equation}
\label{lem_est_Picurlu_eq0}
\| \sqrt{\beta_h}\, \bfPi_K( \bfu - \bfu_I ) \|_{0,K} \lesssim h_K \| \bfu \|_{E,1,K}.
\end{equation}
\end{lemma}
\begin{proof}
Since $\bfPi_K( \bfu - \bfu_I )\in \nabla S_h(K)$, by \eqref{ife_h1_hucl_hdiv} we have a $\varphi_h\in \widetilde{S}^n_h(K)$ from Remark \ref{lem_varphi_stab} such that $\bcurl\,\varphi_h = \beta_h\bfPi_K( \bfu - \bfu_I)$.
Then, integration by parts leads to
\begin{equation}
\begin{split}
\label{lem_est_Picurlu_eq2}
\| \sqrt{\beta_h}\, \bfPi_K( \bfu - \bfu_I ) \|^2_{0,K} & = \int_K \bcurl \, \varphi_h \cdot ( \bfu - \bfu_I)  \dd \bfx  = - \int_{\partial K} \varphi_h( \bfu - \bfu_I)\cdot\bft  \dd s.
\end{split} 
\end{equation}
Next, the H\"older's inequality, estimate \eqref{varphi_stab}, and the definition of $\bfu_I$ together lead to
\begin{equation}
\begin{split}
\label{lem_est_Picurlu_eq3}
\| \sqrt{\beta_h}\, \bfPi_K( \bfu - \bfu_I ) \|^2_{0,K}&  \le \| \varphi_h \|_{0,\partial K} \| ( \bfu - \bfu_I)\cdot\bft  \|_{0,\partial K}  \lesssim h^{1/2}_K  \| \beta_h\bfPi_K( \bfu - \bfu_I) \|_{0,K} \, h^{1/2}_K \| \bfu \|_{1/2,\partial K}.
\end{split}
\end{equation}
Note that $\| \bfu \|_{1/2,\partial K} \lesssim \| \bfu \|_{E,1,K}$. Hence, cancelling one term $\| \sqrt{\beta_h}\, \bfPi_K( \bfu - \bfu_I ) \|_{0,K}$ yields the desired result.
\end{proof}

\begin{lemma}
\label{lem_curl_est_piu}
Let $\bfu\in \bfH^1(\curl,\alpha,\beta;\mathcal T_h)$. Then it follows that
\begin{equation}
\label{lem_curl_est_pi_eq0}
\| \sqrt{\beta_h}\,( \bfu - \bfPi_K\bfu_I ) \|_{0,K} \lesssim h_K \| \bfu \|_{E,\curl,1,\omega_K} +  \| \bfu \|_{E,0,\delta{K}}. 
\end{equation}
\end{lemma}
\begin{proof}
The desired result directly follows from the following decomposition
\begin{equation}
\label{lem_curl_est_pi_eq1}
\| \sqrt{\beta_h}\,( \bfu - \bfPi_K\bfu_I ) \|_{0,K} \le \| \sqrt{\beta_h}\,( \bfu - \bfPi_K\bfu ) \|_{0,K} + \| \sqrt{\beta_h}\,\bfPi_K( \bfu - \bfu_I ) \|_{0,K}
\end{equation}
together with \eqref{Pi_curl_est_eq1} and Lemma \ref{lem_est_Picurlu}.
\end{proof}

\begin{lemma}
\label{lem_curl_est_piu_bd}
Let $\bfu\in \bfH^1(\curl,\alpha,\beta;\mathcal T_h)$. Then it follows that
\begin{equation}
\label{lem_curl_est_piu_bd_eq0}
\| \sqrt{\beta_h}(\bfu -\bfPi_K\bfu_I) \|_{0,\partial K} \lesssim h^{1/2}_K \| \bfu \|_{E,\curl,1,\omega_K}.
\end{equation}
\end{lemma}
\begin{proof}
Similar to \eqref{lem_curl_est_pi_eq1}, we first write
\begin{equation}
\label{lem_curl_est_piu_bd_eq1}
\| \sqrt{\beta_h}\,( \bfu - \bfPi_K\bfu_I ) \|_{0,\partial K} \le \| \sqrt{\beta_h}\,( \bfu - \bfPi_K\bfu ) \|_{0,\partial K} + \| \sqrt{\beta_h}\,\bfPi_K( \bfu - \bfu_I ) \|_{0,\partial K}.
\end{equation}
Then, using a similar trace inequality argument with that in Lemma \ref{lem_est_curlu} and \eqref{Pi_curl_est_eq01} in Lemma \ref{Pi_curl_est} lead to
\begin{equation}
\begin{split}
\label{lem_curl_est_piu_bd_eq2}
\|  \bfu - \bfPi_K\bfu  \|_{0,e}  &\lesssim h^{-1/2}_K \|  \bfu^+_E - \bfPi^+_K\bfu  \|_{0,K}  + h^{1/2}_K |  \bfu^+_E  |_{1,K} \lesssim h^{1/2}_K \| \bfu \|_{E,\curl,1,\omega_K} .
\end{split}
\end{equation}
The estimate of the second term in \eqref{lem_curl_est_piu_bd_eq1} follows from the trace inequality for IFE functions in Lemma \ref{lem_trace_inequa} and Lemma \ref{lem_est_Picurlu}.
\end{proof}

We are ready to present the main theorem in this section.
\begin{theorem}
\label{thm_hcurl_est}
Let $\bfu\in \bfH^1(\curl,\alpha,\beta;\mathcal T_h)$ be the solution to \eqref{inter_PDE} and $\bfu_h$ be the solution to \eqref{IVEM_curl}. Then, 
\begin{equation}
\label{thm_hcurl_est_eq0}
\tnorm{\bfu - \bfu_h }_e \lesssim h_K \| \bfu \|_{\bfH^1(\curl;\cup\, \Omega^{\pm})}.
\end{equation}
\end{theorem}
\begin{proof}
Note the decomposition $\bfu - \bfu_h = \bfxi_h + \bfeta_h$ in \eqref{error_decomp_hcurl}. The estimates on non-interface elements are standard. Using Theorem \ref{thm_err_bound_curl} with the Lemmas \ref{lem_est_curlu}--\ref{lem_curl_est_piu_bd}, we obtain
\begin{equation}
\begin{split}
\label{thm_hcurl_est_eq1}
\tnorm{ \bfeta_h }_e & \lesssim  \sum_{K\in\mathcal{T}^i_h} \left( h^{1/2}_K \| \bfu \|_{E,\curl,1,\omega_K} +  \| \bfu \|_{E,0,\delta{K}}  \right) + \sum_{K\in\mathcal{T}^n_h} h_K \| \bfu \|_{\bfH^1(\curl;K)} 
\\
& \lesssim h^{1/2}_K \| \bfu \|_{E,\curl,1,\Omega_{h_{\Gamma}}} + \| \bfu \|_{E,\curl,1,\Omega_{\delta_0}} + h \| \bfu \|_{\bfH^1(\curl;\cup\, \Omega^{\pm})} \lesssim h \| \bfu \|_{\bfH^1(\curl;\cup\, \Omega^{\pm})},
\end{split}
\end{equation}
where we have used Lemma \ref{strip region} with the estimates for $h_{\Gamma}$ and $\delta_0$. In addition, by the definition of $\tnorm{\cdot}_e$, we have
\begin{equation}
\label{thm_hcurl_est_eq2}
\tnorm{\bfxi_h}_e \lesssim \sum_{K\in\mathcal{T}_h} \| \curl\, \bfxi_h \|_{0,K} +  \| \bfPi_K \bfxi_h \|_{0,K} + \| \bfxi_h  -  \bfPi_K \bfxi_h \|_{0,\partial K} 
\end{equation}
where the estimates of the first two terms follow from \eqref{Pi_curl_est_eq02} in Lemma \ref{Pi_curl_est} and Lemma \ref{lem_est_Picurlu}, respectively. For the last term in \eqref{thm_hcurl_est_eq2}, we notice that
\begin{equation}
\label{thm_hcurl_est_eq3}
\| (\bfu - \bfu_I) - \bfPi_K(\bfu - \bfu_I) \|_{0,\partial K} \le \| \bfu -  \bfu_I  \|_{0,\partial K} + \| \bfPi_K(\bfu - \bfu_I)  \|_{0,\partial K}
\end{equation}
where the estimate of the first term is similar to \eqref{lem_est_Picurlu_eq3}, and the estimate of the second term comes from the trace inequality for IFE functions in Theorem \ref{IFE trace} together with Lemma \ref{lem_est_Picurlu}.
\end{proof}
\RG{\begin{remark}
\label{rem_scaling}
If the ``right" scaling $h$ is used in stabilization \eqref{stab_curl} that induces a discrete $H^{-1/2}$ norm to match the regularity of the trace of an $\bfH(\curl)$ vector field in 2D, then, in the derivation of the a priori error bound in Theorem \ref{thm_err_bound_curl}, one has to use the following estimate:
\begin{equation}
\begin{split}
\label{rem_scaling_eq1}
& \int_{\partial K} \alpha_h ( \curl\, \bfu - \curl\, \bfu_I)(\bfeta_h\cdot \bft - \bfPi_K\bfeta_h\cdot\bft ) \dd s \\
\le & \; h^{-1/2} \| \alpha_h( \curl\, \bfu - \curl\, \bfu_I ) \|_{0,\partial K} h^{1/2} \| \bfeta_h\cdot \bft - \bfPi_K\bfeta_h\cdot\bft \|_{0,\partial K}.
 \end{split}
\end{equation}
Opting for this route, the term $h^{1/2} \| \bfeta_h\cdot \bft - \bfPi_K\bfeta_h\cdot\bft \|_{0,\partial K}$ is a part of the norm $\tnorm{\bfeta_h}_e$. Thus, the term $h^{-1/2} \| \alpha_h( \curl\, \bfu - \curl\, \bfu_I ) \|_{0,\partial K}$ needs to yield an $h$ to deliver the optimal order convergence. However, using \eqref{lem_est_curlu_eq0} in Lemma \ref{lem_est_curlu} to estimate this term will immediately lead to the loss of a further $h^{1/2}$ order convergence, such that the final error estimate is only suboptimal. Furthermore, we highlight that such a trick to achieve optimal convergence highly relies on the property that VEM can obtain coercivity even for an ``underweight" scaling parameter. In contrast, $h^{-1}$ scaling has to be used for the purpose of coercivity (norm equivalence) in some unfitted mesh methods, which causes suboptimal convergence. 
\end{remark}
}

\RG{
\section{Numerical Experiments}
In this section, we present some numerical results to validate the analysis above. Here we focus on the $\bfH(\curl)$ problem, as the main motivation for this work is to address the related non-conformity issue that challenges many unfitted mesh methods \cite{2016CasagrandeHiptmairOstrowski,2016CasagrandeWinkelmannHiptmairOstrowski,2020GuoLinZou}. We consider a domain $\Omega=(-1,1)\times(-1,1)$ with a structured Cartesian triangular mesh. Our test example is borrowed from \cite{2012HiptmairLiZou} where the interface is a circle given by $\Gamma:x^2+y^2=r^2_1$ that cuts $\Omega$ into the inside and outside subdomains denoted by $\Omega^-$ and $\Omega^+$. The exact solution is given by
\begin{equation}
\label{exact_solu}
\bfu =
\begin{cases}
      &  \frac{1}{\alpha^-} \left(\begin{array}{c} \left(  - k_1(r_1^2 - x^2 -y^2)y \right) \\ \left(  -k_1(r_1^2 - x^2 -y^2)x  \right) \end{array}\right) ~~~~ \text{in}~ \Omega^-,  \\
      &  \frac{1}{\alpha^+} \left(\begin{array}{c} \left(  - k_2(r_2^2-x^2-y^2)(r_1^2-x^2-y^2)y  \right) \\  \left( - k_2(r_2^2-x^2-y^2)(r_1^2-x^2-y^2)x \right) \end{array}\right) ~~~~ \text{in}~ \Omega^+.
\end{cases}
\end{equation}
The boundary conditions and the right hand side $\bff$ are calculated accordingly. We set $k_2=20$, $k_1=k_2(r_2^2-r_1^2)$ with $r_1=\pi/5$ and $r_2=1$, and consider the parameters: fixing $\alpha^-=\beta^-=1$ and varying $\alpha^+ = \beta^+ = 10$ or $100$. We present the numerical results in the following Figure \ref{fig:error} which clearly show an optimal convergence and outperform many other unfitted mesh methods in the literature.
\begin{figure}[htbp]
\begin{center}
\begin{subfigure}[b]{0.49\linewidth}
      \centering
      \includegraphics[width=0.8\textwidth]{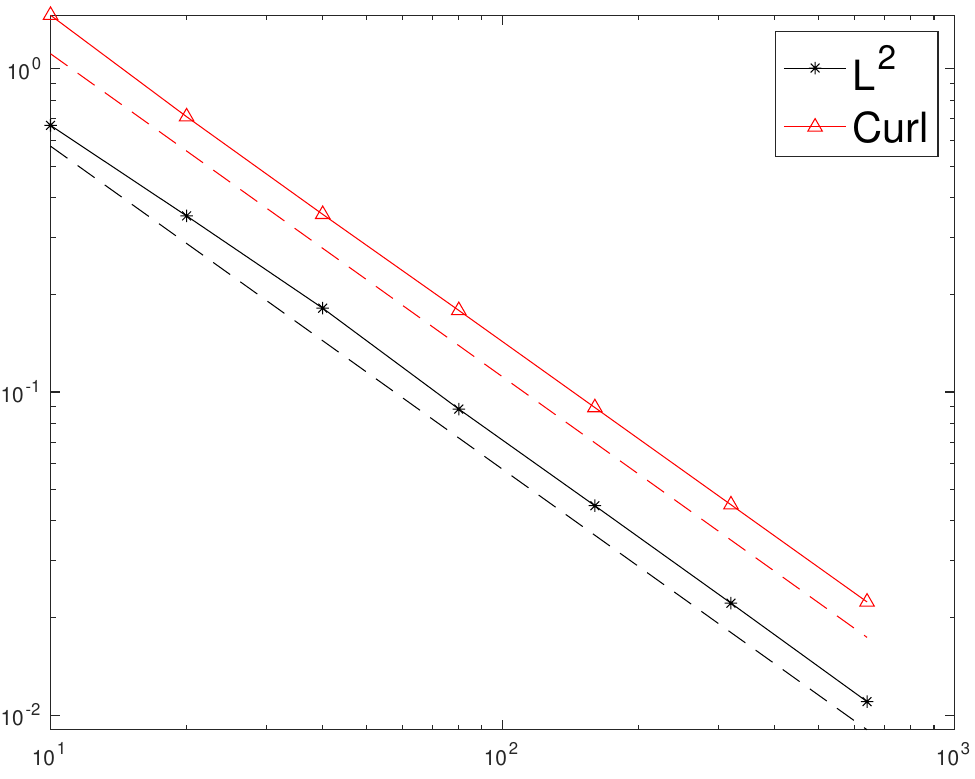}
      \caption{}
      \label{fig:err10}
\end{subfigure}
\begin{subfigure}[b]{0.49\linewidth}
      \centering
      \includegraphics[width=0.8\textwidth]{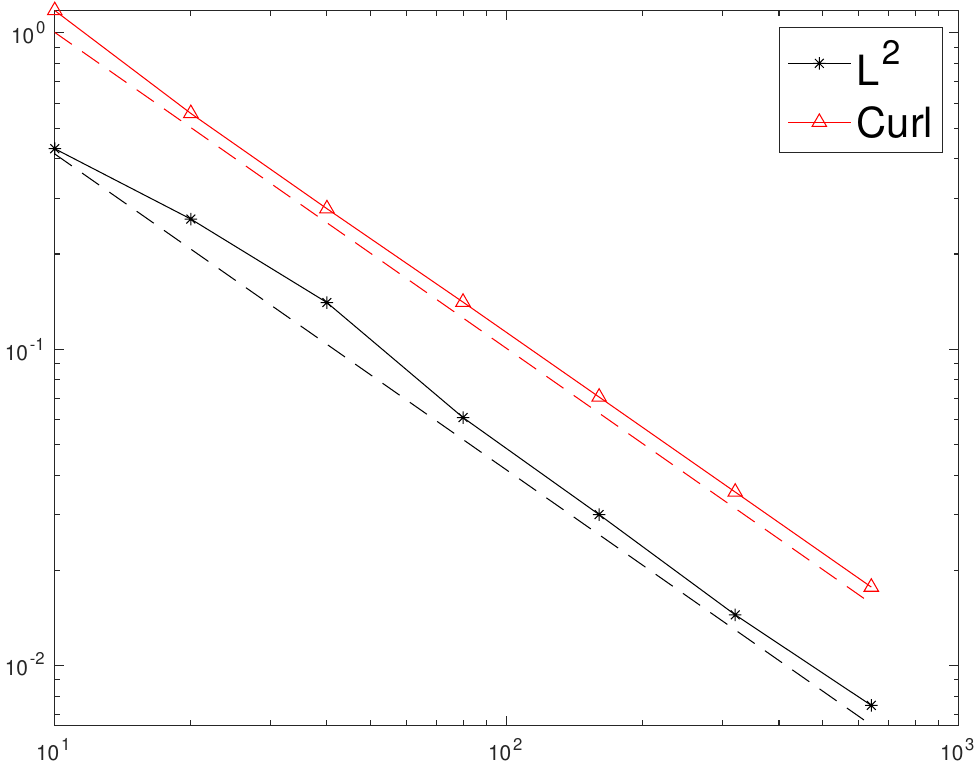}
      \caption{}
      \label{fig:err100}
\end{subfigure}
\end{center}
\caption{Errors of the IVE for the $\bfH(\curl)$ interface problem: $\alpha^+ = \beta^+ = 10$ (left) and $\alpha^+ = \beta^+ = 100$ (right). The dashed lines are the reference lines indicating an optimal convergence of order $\mathcal{O}(h)$.}
\label{fig:error}
\end{figure}
}

\section{Concluding Remarks}

We have developed IVE methods for solving $H^1$ and $\bfH(\curl)$ elliptic interface problems \LC{in two dimensions}. Conventional finite element spaces are conforming but do not satisfy the jump conditions, while the IFE spaces in current literature satisfy the jump conditions but are not conforming. The proposed IVE spaces are conforming and satisfy the jump conditions simultaneously. In our opinion, they are candidates for the ``ideal'' spaces to solve interface problems. This unique attribute makes the proposed methods inherit the advantages of both fitted and unfitted mesh methods. Similar to the classic VEM, the newly constructed spaces are projected to the IFE spaces which is computable directly through DoFs.

There are several major differences of the proposed IVEM from the classic IFEM. First, the proposed method does not require those DG-like edge terms originated from integration by parts. The only edge-based term is the stabilization term. Opposing to IPDG-like methods that has the symmetry-coercive dilemma, what is even more favorable about IVEM is that the resulting discretization is parameter-free, and yields a symmetric system which can be solved by fast linear solvers. This is particularly useful for the $\bfH(\curl)$ case, since it avoids using $h^{-1}$ scaling in the stabilization that causes a loss of convergence order for non-conforming methods~\cite{2016CasagrandeHiptmairOstrowski,2016CasagrandeWinkelmannHiptmairOstrowski,2020GuoLinZou}. Second, the stabilization is completely local, and consequently the assembling does not need to compute the interaction between two neighbor elements' DoFs. This trait makes this method more parallelizable. In addition, there are more DoFs locally on each interface element than classic IFEM, and these extra DoFs are introduced by the cutting points which can better resolve the geometry.

The proposed method is also distinguished from the classic VEM in the fact that anisotropic elements cut by the interface are treated together as a shape regular element. Thanks to this treatment and the properties of IFE spaces, the robust error analysis with respect to cutting points can be achieved which is also much easier and more systematic. In fact, for the analysis of classical VEM on anisotropic elements~\cite{2018CaoChen,2021CaoChenGuo}, the main difficulty is to obtain an error bound that is independent of element anisotropy such as shrinking elements. We highlight that one of the key obstacles for anisotropic analysis is the failure of the standard trace inequalities as the height of an edge may be very small and thus unable to support a smooth extension of a function defined on an edge toward the interior. For example for the present situation, in the estimation of \eqref{projection term estimate eq3} and \eqref{projection_interpolation error on edge eq3}, the standard trace inequality cannot be applied directly to each polynomial on each subelement as it may shrink, and thus the hidden constant may not be uniform with respect $h$ anymore. Consequently, the estimation for VEM generally requires some dedicated analysis techniques such as the Poincar\'e inequality on an anisotropic cut element developed in~\cite{2018CaoChen,2021CaoChenGuo}. This is especially difficult for the $\bfH(\curl)$ case that demands a virtual mesh, see \cite{2021CaoChenGuo}. These specialized analysis may limit the scope of its applicable elements. 
However, in the proposed analysis of this paper, these special treatments are not needed anymore. This improvement comes from the benefit of adopting the piecewise polynomial IFE functions as our projection space, which do admit cutting geometry-independent trace inequalities on interface elements as shown in Lemmas \ref{IFE trace} and \ref{lem_trace_inequa}. These trace inequalities significantly simplify the analysis, which are now streamlined to resemble more to the standard analysis on isotropic elements.

\newrevision{Similar to many unfitted mesh methods in the literature, the present analysis relies on that the interface is smooth. If the interface is non-smooth (even piecewise smooth), many critical tools for the analysis will not be available anymore. For example, if the interface has geometrical singularities, the solutions will have lower regularity (\cite{1990NicaisePolygonal}). Consequently, (i) $H^2$ and $\bfH^1(\curl)$ Sobolev extensions become obscure, (ii) commuting diagrams with extra smoothness in Section \ref{subsec:SpaceNorms} do not hold anymore.}

\LC{We focus on two-dimensional problems in this work to introduce the methodology, which can shed light on the 3D case. In a more recent work \cite{2022CaoChenGuo}, the IVE spaces and the schemes are extended to the 3D case. As one can imagine, the definition of IVE and IFE spaces as well as anisotropic error analysis in 3D will be much more complicated.}

\begin{acknowledgements}
The authors are grateful for the constructive advice from the anonymous reviewers. 
\end{acknowledgements}

\end{document}